%% file: template-article.tex
\documentclass[a4paper,oneside,10pt,oldfontcommands,reqno]{amsart}
\usepackage[a4paper, total={426pt, 674pt}]{geometry}

\newcommand{\langue}{anglais}	

\usepackage{ifthen}
\ifthenelse{\equal{\langue}{francais}}{
	\usepackage[french]{babel}}
	{\ifthenelse{\equal{\langue}{anglais}}{\usepackage[english]{babel}}{}}
\usepackage[T1]{fontenc}
\usepackage[utf8]{inputenc}
\usepackage[babel,german=guillemets]{csquotes}
\usepackage{eso-pic}
\usepackage[foot]{amsaddr}
\usepackage[backend=biber, style=alphabetic, sorting=nyt, giveninits=true, doi=false, isbn=false, url=false, eprint=false, maxbibnames=99, maxcitenames=5]{biblatex} 
\DeclareNameAlias{author}{last-first}
\usepackage{orcidlink}

\usepackage{lmodern} 
\usepackage{enumitem}
\usepackage{ifthen}
\ifthenelse{\equal{\langue}{francais}}{
	\frenchbsetup{StandardLists=true}}
	{}

\usepackage{graphicx} 
\usepackage{subcaption}

\usepackage{amsmath}
\usepackage{amssymb}
\usepackage{amsthm}
\usepackage{amsfonts}
\usepackage{bbold}
\usepackage{mathtools}
\usepackage{tikz-cd}
\usepackage{hyperref}
\usepackage{multirow}
\usepackage[normalem]{ulem}
\usepackage{cases}

\ifthenelse{\equal{\langue}{francais}}{
	\newcommand{\theoremenom}{Théorème}
	\newcommand{\propositionnom}{Proposition}
	\newcommand{\lemmenom}{Lemme}
	\newcommand{\corollairenom}{Corollaire}
	\newcommand{\definitionnom}{Définition}
	\newcommand{\remarquenom}{Remarque}
	\newcommand{\exemplenom}{Exemple}
	\newcommand{\conjecturenom}{Conjecture}
}{\ifthenelse{\equal{\langue}{anglais}}{
	\newcommand{\theoremenom}{Theorem}
	\newcommand{\propositionnom}{Proposition}
	\newcommand{\lemmenom}{Lemma}
	\newcommand{\corollairenom}{Corollary}
	\newcommand{\definitionnom}{Definition}
	\newcommand{\remarquenom}{Remark}
	\newcommand{\exemplenom}{Example}
	\newcommand{\conjecturenom}{Conjecture}
}{}}
	
\newtheorem{theoreme}{\theoremenom}[section]
\newtheorem{proposition}[theoreme]{\propositionnom}
\newtheorem{lemme}[theoreme]{\lemmenom}
\newtheorem{corollaire}[theoreme]{\corollairenom}
\newtheorem{definition}[theoreme]{\definitionnom}
\newtheorem{remarque}[theoreme]{\remarquenom}

\newtheorem*{conjecture}{\conjecturenom}

\usepackage{thmtools}
\usepackage{thm-restate}

\makeatletter
\def\cleartheorem#1{%
    \expandafter\let\csname#1\endcsname\relax
    \expandafter\let\csname c@#1\endcsname\relax
}
\makeatother
\newcommand{\compteurThm}{1}
\newcounter{annexe}

\newcommand{\R}{\mathbb{R}}
\newcommand{\N}{\mathbb{N}}
\newcommand{\dd}{\mathcal{A}}

\newcommand{\eqskip}{ \vspace*{2mm}\\ }

\DeclareUnicodeCharacter{0301}{\'{e}}
\DeclareMathOperator{\dist}{dist}
\DeclareMathOperator{\dv}{div}

\begin{document}

\pagestyle{empty} 


\title{Payne's nodal line conjecture fails on doubly-connected planar domains}

\author[P. Freitas]{Pedro Freitas\orcidlink{0000-0002-2007-5259}}
\author[R. Leylekian]{Roméo Leylekian\orcidlink{0000-0002-9277-0437}}
\address{Grupo de F\'{\i}sica Matem\'{a}tica, Instituto Superior T\'{e}cnico, Universidade de Lisboa, Av. Rovisco Pais, 1049-001 Lisboa, Portugal}
\email{pedrodefreitas@tecnico.ulisboa.pt, romeo.leylekian@tecnico.ulisboa.pt}



\begin{abstract}

We present examples of bounded planar domains with one single hole for which the nodal line of a second Dirichlet eigenfunction is closed
and does not touch the boundary. This shows that Payne’s nodal line conjecture can at most hold for simply-connected domains in the plane.

\end{abstract}

\maketitle



\pagestyle{plain} 


\input{parties/intro}

\input{parties/contrexemple}





\begin{appendix}
\input{parties/annexe} 
\end{appendix}

\section*{Acknowledgements}
We would like to thank the {organisers} of the INDAM workshop \enquote{Matter or Shapes} held in Cortona in May 2025, and D.
Krej\v{c}i\v{r}\'{\i}k, whose talk during the session of open problems at the conference renewed our interest in Payne's nodal
line Conjecture. This work was partially supported by the Funda\c{c}\~ao para a Ci\^encia e a Tecnologia, Portugal, via the research centre
GFM, reference UID/00208/2025.

\printbibliography



\end{document}

%% file: parties/intro.tex
\section{Introduction}

Consider the eigenvalue problem for the Dirichlet Laplacian on a bounded connected domain $\Omega$ in $\R^n$, namely,
\begin{equation}\label{eigprob}
\left\{
\begin{array}{ll}
  \Delta u + \lambda u = 0, & x\in\Omega\eqskip
  u=0, & x\in\partial\Omega,
\end{array}
\right.
\end{equation}
with eigenvalues $\lambda_{1}<\lambda_{2}\leq\lambda_{3} \leq \dots $ and corresponding eigenfunctions $u_{1}, u_{2}, u_{3}\dots$.
While the first eigenfunction $u_{1}$ may be chosen to be of one sign in $\Omega$, due to orthogonality conditions and Courant's nodal 
domain theorem, any second eigenfunction $u_{2}$ will have two and only two nodal domains, defined as maximal connected open sets in
$\Omega$ where an eigenfunction does not change sign. More precisely, the nodal set of a function $u:\Omega\to\R$ is defined by
\[
 \mathcal{N}(u) = \overline{\left\{ x\in\Omega: u(x)=0\right\}},
\]
and the nodal domains of $u$ are the connected components into which the domain $\Omega$ is divided by $\mathcal{N}(u)$.

The structure of the nodal set and nodal domains of the Laplacian has been the object of much study, with the
case associated to the second eigenfunction being of particular interest. This had at its origin Payne's nodal line conjecture
concerning the topology of $\mathcal{N}(u_{2})$, and was first formulated for planar domains in~\cite{Payne} as
\begin{conjecture}[Nodal line conjecture 1967]
 The eigenfunction $u_{2}$ of Problem~\eqref{eigprob} cannot have a closed nodal line for any domain $\Omega\subset\R^{2}$.
\end{conjecture}
\noindent An $n-$dimensional version was proposed by Yau in 1993~\cite[Problem 45]{yau93}, who also questioned how 
Euclidean results would extend to compact manifolds with boundary.

The original conjecture has been shown to hold under certain restrictions, namely by Payne when $\Omega$ is convex along one
direction and invariant by reflection with respect to an axis orthogonal to that direction~\cite{Payne2}, and by Melas for sufficiently smooth
convex planar domains~\cite{Melas}. The smoothness assumption in Melas's paper was then removed in~\cite{alessandrini}. Damascelli generalised Payne's result to higher dimensions~\cite{d00}, and Jerison proved
the conjecture in the case of long and thin convex domains~\cite{jerison}, while  Krej\v{c}i\v{r}\'{\i}k and the first author proved
it for long thin tubes of arbitrary cross-section around a bounded line in Euclidean space~\cite{fk08}. In these last two cases,
the problem of locating the nodal set was also considered. More recently, Kiwan addressed the conjecture for a class of doubly-connected domains~\cite{kiwan}, while Mukherjee and Saha proved it for dumbbell-like domains in any dimension~\cite[Theorem~1.8]{mukherjee-saha}. Let us also mention that, in the preprint~\cite{ha24}, the author claims that the conjecture holds for all regular simply connected domains. However, the proof relies on a result in~\cite{linqun} according to which the multiplicity of the second eigenvalue on planar domains is at most two.
As recalled in~\cite{H2ON-2,helffer-hoffmann-ostenhof-jauberteau-lena} this need not be the case in general.

On the negative side, several counterexamples have been provided, either for the original problem or for related problems under
different  settings. In the former category, the first known counterexample was given in~\cite{H2ON} (see also \cite{H2ON-2}) and consists of a disk from
which a (large) number of slits located on a smaller concentric disk have been removed.
Later Fournais generalised this to higher dimensions giving a domain where the nodal set does not touch the boundary,
while providing a first estimate on the order of the number of holes, namely, $10^9$, which was considered to be {\it probably
by far too large}~\cite{fournais01}. In fact, Kennedy showed that in dimensions three or higher it was possible to obtain a 
counterexample homeomorphic to a ball~\cite{kennedy} (see also~\cite[Theorem~1.9]{mukherjee-saha}).

As for what may happen under different restrictions, a first question that may be asked is what happens if one drops the
assumption that the domain is bounded. In this case, a counterexample relevant to the present paper -- see the discussion
below -- was given by Krej\v{c}i\v{r}\'{\i}k and the first author in~\cite{freitas-krejcirik07}, where it was shown that
there exist unbounded domains for which the nodal line does not touch the boundary. These are however not closed and in
fact consist of an infinite straight line. On the other hand, these domains satisfy Payne's conditions mentioned above for
which he proved the conjecture for bounded domains~\cite{Payne2}. In fact, the counterexample given in~\cite{freitas-krejcirik07}
may even be chosen to be symmetric and convex with respect to both directions.

Moving now away from the Euclidean Laplacian, Lin and Ni showed that in the presence of a potential the conjecture can
fail even for a ball and when the potential is radial, meaning that the nodal set can be an $(n-1)-$sphere~\cite{LinNi}.
Answering the question raised by Yau mentioned above, the first author of the present paper showed that the Laplacian could
also have a nodal set consisting of a geodesic disc on surfaces of revolution~\cite{Freitas02}. These two examples showed that
outside the Laplacian in the Euclidean setting, simply-connectedness and convexity of the domain do not necessarilly imply the 
conjecture. We note that what may happen for the Laplacian in general convex domains in $\R^{n}$ with $n$ larger than two
remains mostly unknown.

All of this leaves open the question of what is the minimal number of holes which a planar counterexample can
have. This had already been formulated in~\cite[Remark~3]{H2ON}, where the authors stated that while they did not believe that it
would be possible to find a planar simply-connected domain for which the conjecture would fail, a more general question would be:

\medskip

{\noindent\it What is the smallest $N_{0}$ such that there exists a domain with $N_{0}$ boundary
components whose second eigenfunction has a nodal line that does not hit the boundary?}

\medskip

In a recent paper, Dahne, Gómez-Serrano,
and Hou, have shown via a computer-assisted proof that it is possible to reduce the number of holes to six (and hence $N_{0}$ above to seven)~\cite{dgh21}. The main purpose
of the present paper is to show through purely analytic arguments that $N_{0}$ may be brought down to two, that is, there exists a planar
domain with a single hole whose nodal line is closed and does not touch the boundary. In other words, Payne's conjecture in
the plane can be expected to hold at most for simply-connected domains.

\begin{theoreme}\label{thm:resultat principal}
There exists a doubly-connected bounded open subset of $\R^2$ with a second eigenfunction whose nodal line is closed and does not touch the boundary.
\end{theoreme}

Of course, as a consequence of Theorem~\ref{thm:resultat principal} there are domains with more holes for which the nodal line remains closed. This can be seen by introducing
sufficiently small holes inside the nodal domains of the eigenfunction of the theorem (see also~\cite[Section~2]{mukherjee-saha}).

\begin{corollaire}\label{corollaire:plus de trous}
For any integer $N\geq 2$, there exists a bounded connected open subset of $\R^2$ whose boundary is made of $N$ connected components, and with a second eigenfunction whose nodal line is closed and does not touch the boundary.
\end{corollaire}

\begin{remarque}
By construction, the domain of Theorem~\ref{thm:resultat principal} (and that of Corollary~\ref{corollaire:plus de trous}) can be assumed to have further properties. For instance, one can take it to be smooth, symmetric, and with a simple second eigenvalue. We also point out that, through small deformations, one shall construct a continuum of domains, including non-symmetric ones, whose second eigenfunction still has a closed nodal line.
\end{remarque}

Let us now explain the ideas behind the proof of Theorem~\ref{thm:resultat principal}. While most, if not all, of the counterexamples to the original conjecture which have appeared in the literature so far have at
their origin that in~\cite{H2ON}, the present example is inspired by that in~\cite{freitas-krejcirik07}. 
This is basically obtained by starting from a long and thin convex domain and appending two thinning bisecting semi-infinite
vertical strips to the parts of the boundary with smallest curvature -- see~\cite{freitas-krejcirik07} for details. In this setting,
it is proved  that when the strips are suffciently narrow the nodal line coincides with the vertical axis bisecting the domain, and
hence never touches the boundary. However, the resulting  domain is unbounded and the nodal line is not closed.

The main idea we shall develop here to produce a bounded counterexample with a closed nodal line is to attach a thin
annulus-like appendix to the original set instead of the two semi-infinite strips, so that the resulting domain
remains bounded and symmetric with respect to the horizontal axis -- see Figure~\ref{fig:rectangle}, where we have represented
schematically possible different configurations for domains of this type, together with their nodal line. Note that as in the figure, such a domain may be the difference of two sets that are symmetric with respect to the horizontal axis and convex in the horizontal direction, a setting which resembles (yet differs
from) that of~\cite{kiwan}, where Payne's conjecture was shown to hold. As explained below,
if the annulus is located appropriately, the nodal line remains trapped inside it, and hence is closed. Of course that in such
a setting it would be very hard if not altogether 
impossible to determine the exact location  of the annulus for which such a configuration occurs. To circumvent this issue, we will use a continuity 
argument and develop what we call a {\it sliding procedure}. More precisely, let us place the annulus so that it intersects the convex 
domain (which for simplicity we now take to be a rectangle) and move it continuously from the left to the right (see 
Figure~\ref{fig:rectangle}). The resulting domains have a boundary with two connected components which we will call the
left and right boundary components, depending on whether they contain the left- and right-most vertices of the rectangle, respectively.
Let us recall that the nodal line of the second eigenfunction over the rectangle coincides with the vertical segment bisecting the
rectangle (the dotted line in Figure~\ref{fig:rectangle}). When the annulus is very thin the nodal line of the second 
eigenfunction over the domain constructed (in blue in Figure~\ref{fig:rectangle}) must be close in some sense to  this bisecting
segment. Hence when the annulus is placed so that the intersection with the rectangle lies on the left of
the bisecting segment, the nodal line will likely touch only the right boundary component, as in 
Figure~\ref{fig:rectangle_t=t--}. And when the annulus intersects the rectangle on the right of the bisecting segment, the nodal
line will touch only the left boundary component, cf. Figure~\ref{fig:rectangle_t=t++}.

\begin{figure}[h]
    \centering
    \begin{subfigure}[b]{0.3\textwidth}
        \centering
        \includegraphics[height=.2\textheight]{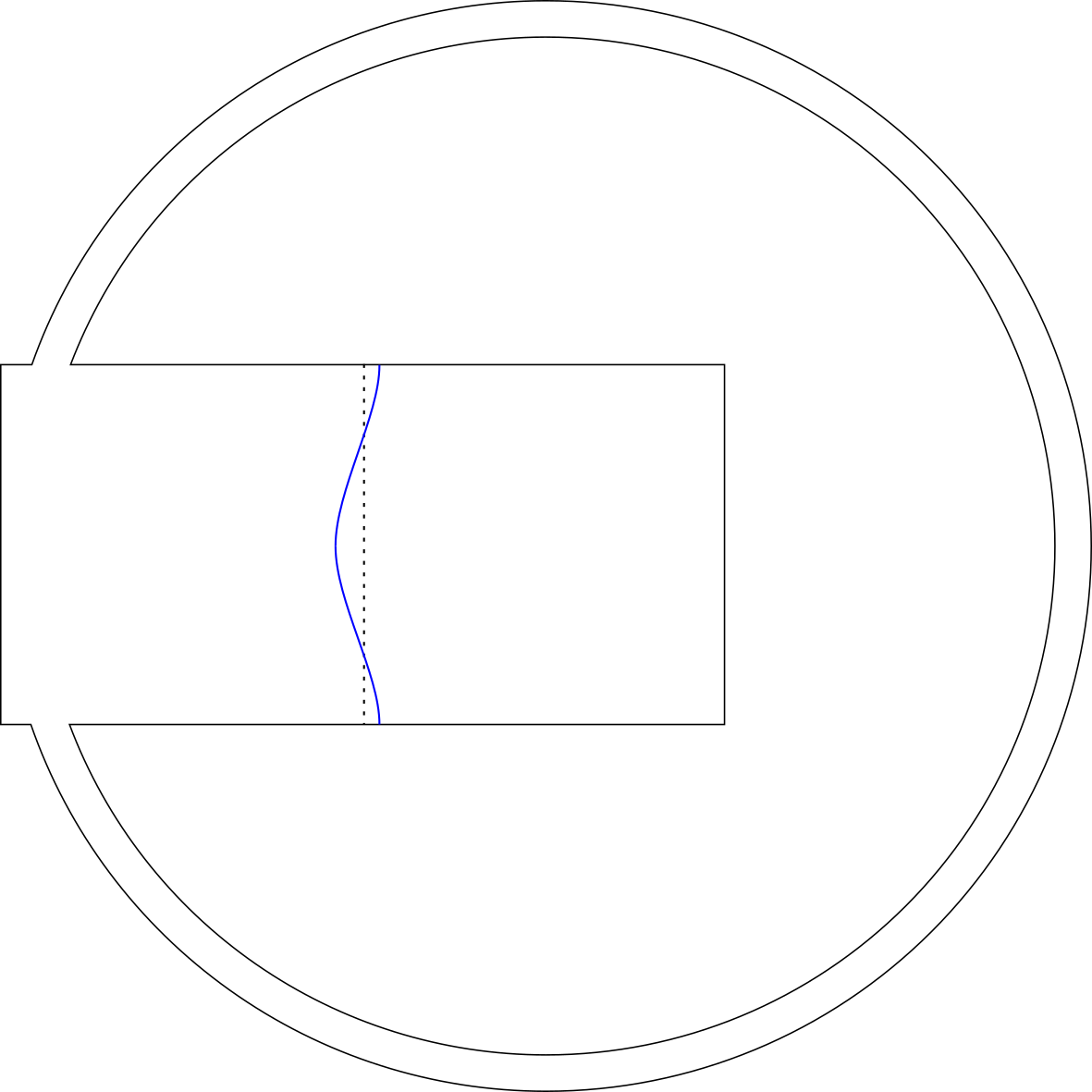}
        \caption{}
        \label{fig:rectangle_t=t--}
    \end{subfigure}
    \hspace{1em}
    \begin{subfigure}[b]{.3\textwidth}  
        \centering 
        \includegraphics[height=.2\textheight]{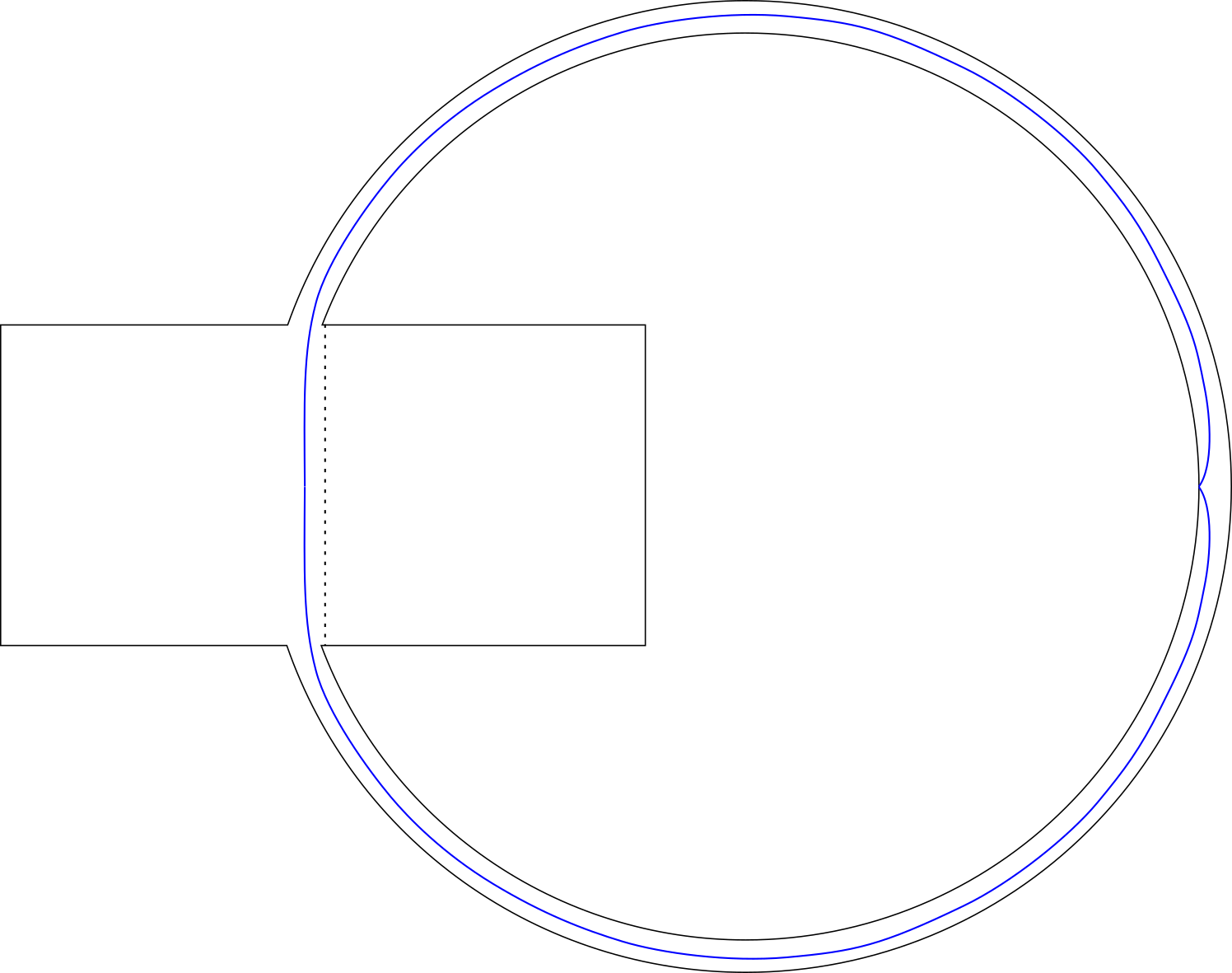}
        \caption{}
        \label{fig:rectangle_t=t-}
    \end{subfigure}
    \vskip\baselineskip
    \begin{subfigure}[b]{0.3\textwidth}   
        \centering 
        \includegraphics[height=.2\textheight]{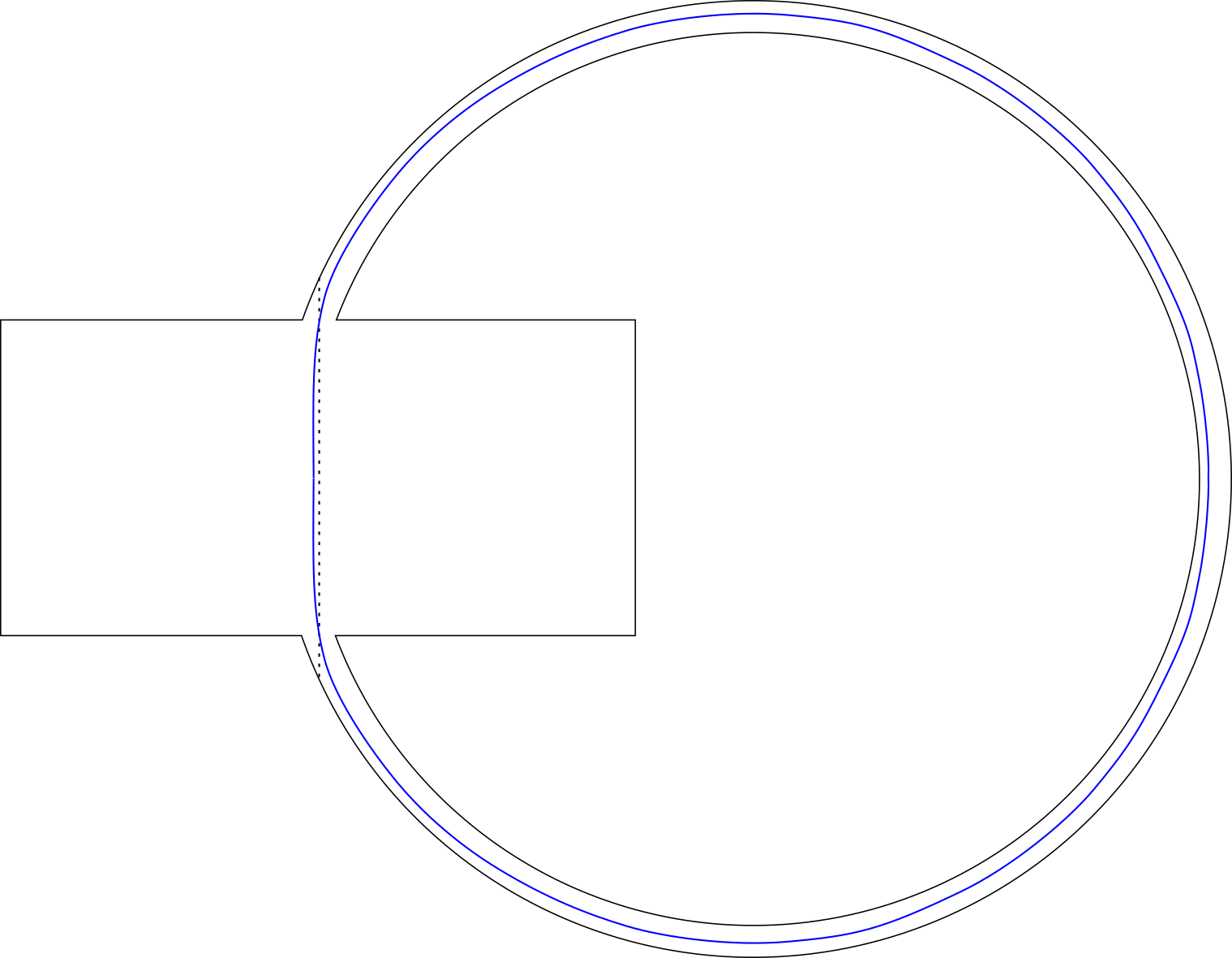}
        \caption{}
        \label{fig:rectangle_t=t0}
    \end{subfigure}
    \hspace{5em}
    \begin{subfigure}[b]{0.3\textwidth}   
        \centering 
        \includegraphics[height=.2\textheight]{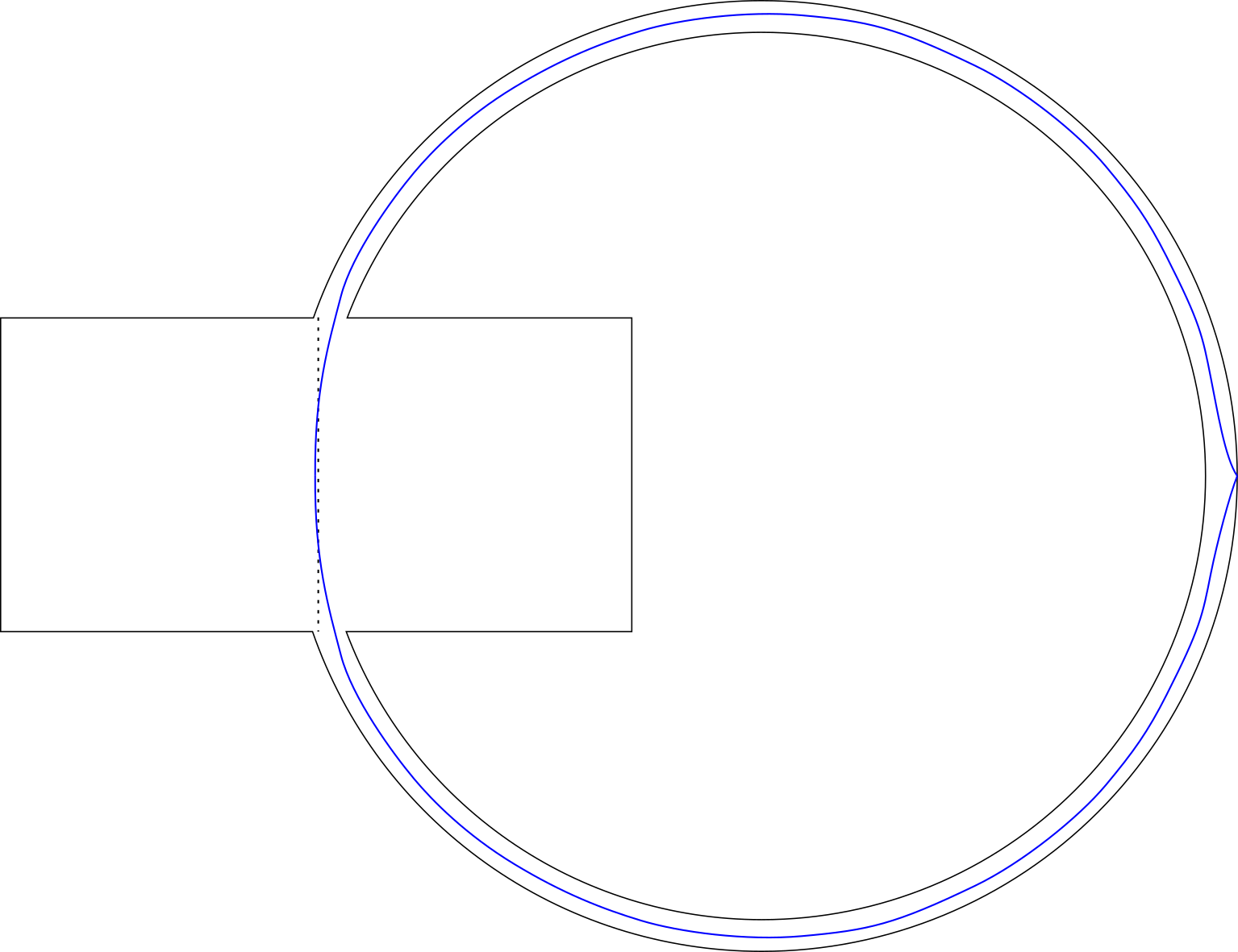}
        \caption{}
        \label{fig:rectangle_t=t+}
    \end{subfigure}
    \vskip\baselineskip
    \begin{subfigure}[b]{0.3\textwidth}  
        \centering 
        \includegraphics[height=.2\textheight]{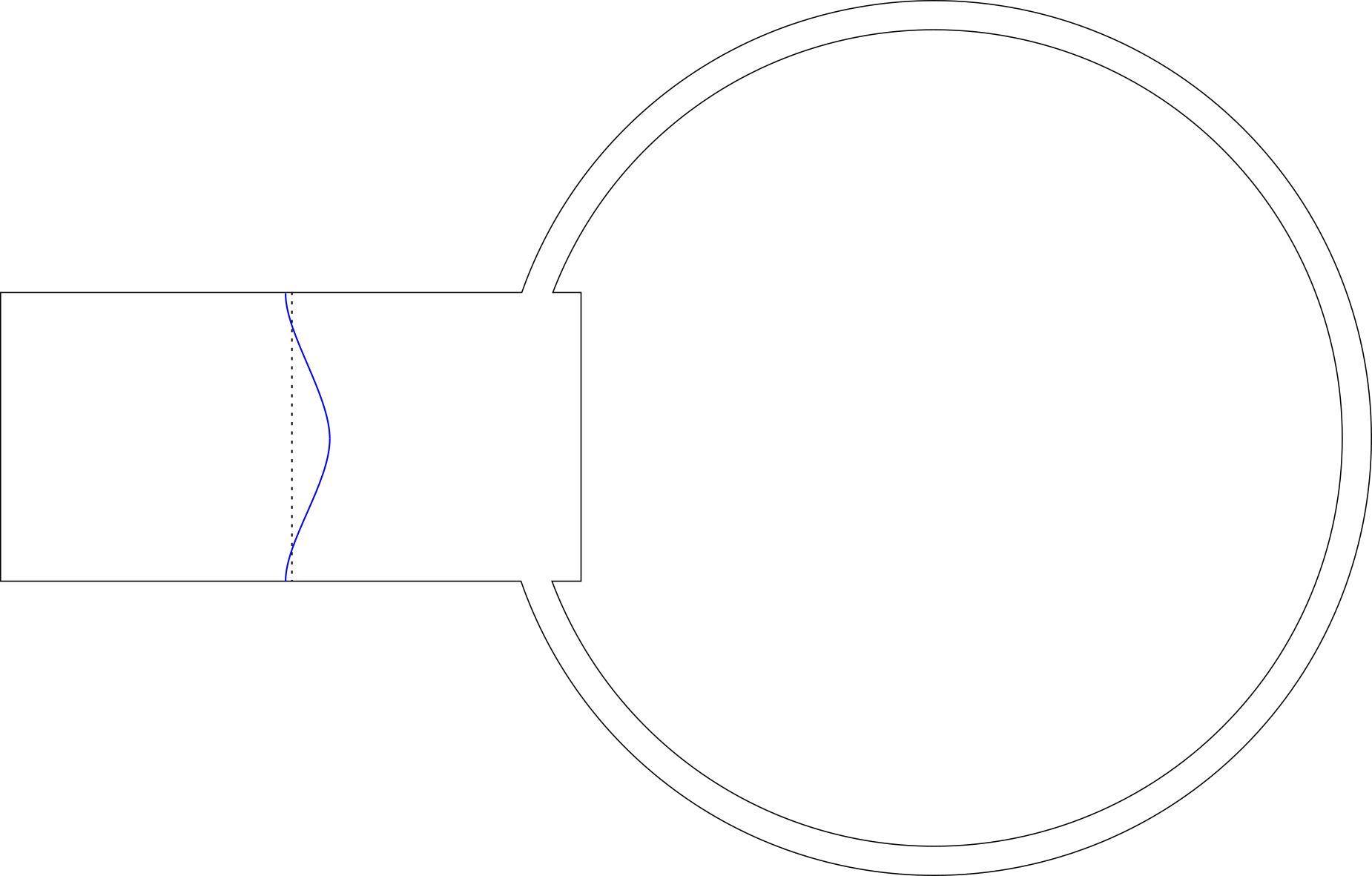}
        \caption{}
        \label{fig:rectangle_t=t++}
    \end{subfigure}
    \caption{The family of doubly-connected domains described in the introduction, in different configurations. In each case, the
    dotted line is the segment bisecting the rectangle and the blue line is the expected profile of the nodal line.}
    \label{fig:rectangle}
\end{figure}

Furthermore, as one moves the annulus continuously from the left to the right the nodal line should also move continuously in some sense.
For instance, we expect that the intersection points between the nodal line and the boundary of the domain, that we shall call impact points below, evolve continuously
on the boundary. This continuity result is made formal in Lemma~\ref{lemme:continuité ligne nodale}. In particular, the
impact points cannot jump from the right to the left boundary component as we move from the situation of
Figure~\ref{fig:rectangle_t=t--} to the situation of Figure~\ref{fig:rectangle_t=t++}. Actually, the only way for those
points to switch from one component to another is to eventually merge, as in Figure~\ref{fig:rectangle_t=t-}. This is the only
possibility for the nodal line to detach from one boundary component (Figure~\ref{fig:rectangle_t=t0}) and then move continuously
toward the other boundary component (Figure~\ref{fig:rectangle_t=t+}). In particular, during the switch from the right to the 
left component, the nodal line remains closed inside the domain, which proves the claim.

So as not to add further technicalities which are beyond the scope of the paper, we will not study the evolution of the nodal line
in this level of detail, and we will only exhibit some minimal properties ensuring that the nodal line is closed at some point (see for instance Lemma~\ref{lemme:continuité ligne nodale}, where only some lower semi-continuity property of the nodal line with respect to the deformation is proved). Nevertheless, at that special point, the localisation of the nodal line described above and schematically represented in Figure~\ref{fig:rectangle_t=t0} is quite accurate, in the sense that the nodal line must indeed be trapped within the annulus and has to enclose the hole of the domain (see Remarks~\ref{remarque:localisation de la ligne nodale} and~\ref{remarque:sign le long du segment}). We do not claim that this particular configuration holds for any doubly-connected domain for which Payne's conjecture fails.

The previous construction is fairly simple and quite general, in the sense that it does not rely on the geometry of the rectangle and the
annulus in a rigid way. Indeed, instead of the rectangle one could take other kinds of domains symmetric with respect to the $x$-axis. 
Similarly, instead of the annulus, one could take a thin tubular domain, symmetric with respect to the $x$-axis, and let it slide from 
left to right, as described above. We will call the resulting one-parameter family of domains of this type a {\it family of
domains with a sliding handle} (see Definition~\ref{def:sliding domains}). Proving that the nodal line on domains belonging to 
such a family behaves as expected will be the core part in producing our counterexample. This will require several technical details, beginning with Definition~\ref{def:sliding domains}, where in particular some rather strong continuity with respect to the sliding parameter is demanded on the eigenfunctions. It then remains to show that there are indeed families of domains satisfying all the conditions in
Definition~\ref{def:sliding domains}. This will be done starting from the above construction based on a rectangle and a moving annulus. 
However note that in this way one has to handle simultaneously two parameters: one responsible for the width and the other for the translation of
the annulus. This makes the construction quite delicate. Furthermore we have to ensure the strong continuity property for the eigenfunctions, yielding an additional technical degree to the proofs. That is why we will smoothen the domains at the corners (see Figure~\ref{fig:famille_reguliere}).

In Section~\ref{sec:general}, we define the notion of a family of domains with a sliding handle and state the general results
from which Theorem~\ref{thm:resultat principal} follows. Section~\ref{sec:ligne nodale famille doublement connexe}
is dedicated to proving that there exists at least one domain in such a family which has a closed nodal
line. Finally, we show in Section~\ref{sec:existence famille doublement connexe} that a family of domains with a sliding handle does exist.

%% file: parties/contrexemple.tex
\section{Families of domains with a sliding handle}\label{sec:general}

Theorem~\ref{thm:resultat principal} will follow from a general result stating that the property of having a closed nodal
line inside the domain holds for any domain built through the sliding procedure described in the introduction. The following
definition lists the requirements on the family of domains that will be needed in order to ensure that the sliding procedure
does produce an eigenfunction with this property. Among those requirements, some uniform continuity (and not just an $L^2$-continuity) of the eigenfunctions with respect to the sliding parameter is required. This will allow us to prove a minimal regularity in the evolution of the nodal line (see Lemma~\ref{lemme:continuité ligne nodale}). To express this stronger continuity assumption it is necessary to pull back the eigenfunctions in a fixed domain, as explained in the Appendix, and this is why we introduce the annulus $\dd$ defined
by $\dd:=B_2\setminus \overline{B_1}$. The connected components $\partial B_1$ and $\partial B_2$
of its boundary will be called the inner and outer boundaries of $\dd$, respectively.

\newcommand{\fdsh}{\textup{FDSH}}
\begin{definition}\label{def:sliding domains}
Let $T$ be a non-degenerate interval. A one-parameter family of bounded connected open subsets $(\Omega_t)_{t\in T}$ of $\R^2$ is called a family of domains with a sliding handle, \fdsh\/ for short, whenever:
\begin{enumerate}
\item For all $t\in T$, $\Omega_t$ is symmetric with respect to the $x$-axis and homeomorphic to $\dd$.
\item For all $t\in T$, the second eigenvalue on $\Omega_t$ is simple.
\item\label{it:signe segment} There exists a segment $\sigma$ contained in the $x$-axis such that for all $t\in T$, $\sigma\subset\Omega_t$ and the second eigenfunction of $\Omega_t$ changes sign along $\sigma$.
\item The eigenfunctions are uniformly continuous on $t$, in the sense that for each $t\in T$ there exists a second eigenfunction $u_t$ on $\Omega_t$ and a pull-back $\Phi_t:\dd\to\Omega_t$ preserving the symmetry with respect to the $x$-axis, for which the pulled-back eigenfunction $v_t:=u_t\circ\Phi_t$ is continuous with respect to $t\in T$ uniformly with respect to $x\in \overline{\dd}$. In other words,
$T\ni t\mapsto v_t\in C(\overline{\dd})$
is continuous.
\item There exists $t_-\in T$ and $t_+\in T$ such that the nodal line of $v_{t_-}$ does not touch the outer boundary of $\dd$ and the nodal line of $v_{t_+}$ does not touch the inner boundary of $\dd$.
\end{enumerate}
\end{definition}

\begin{figure}[h]
\includegraphics[height=.3\textwidth]
{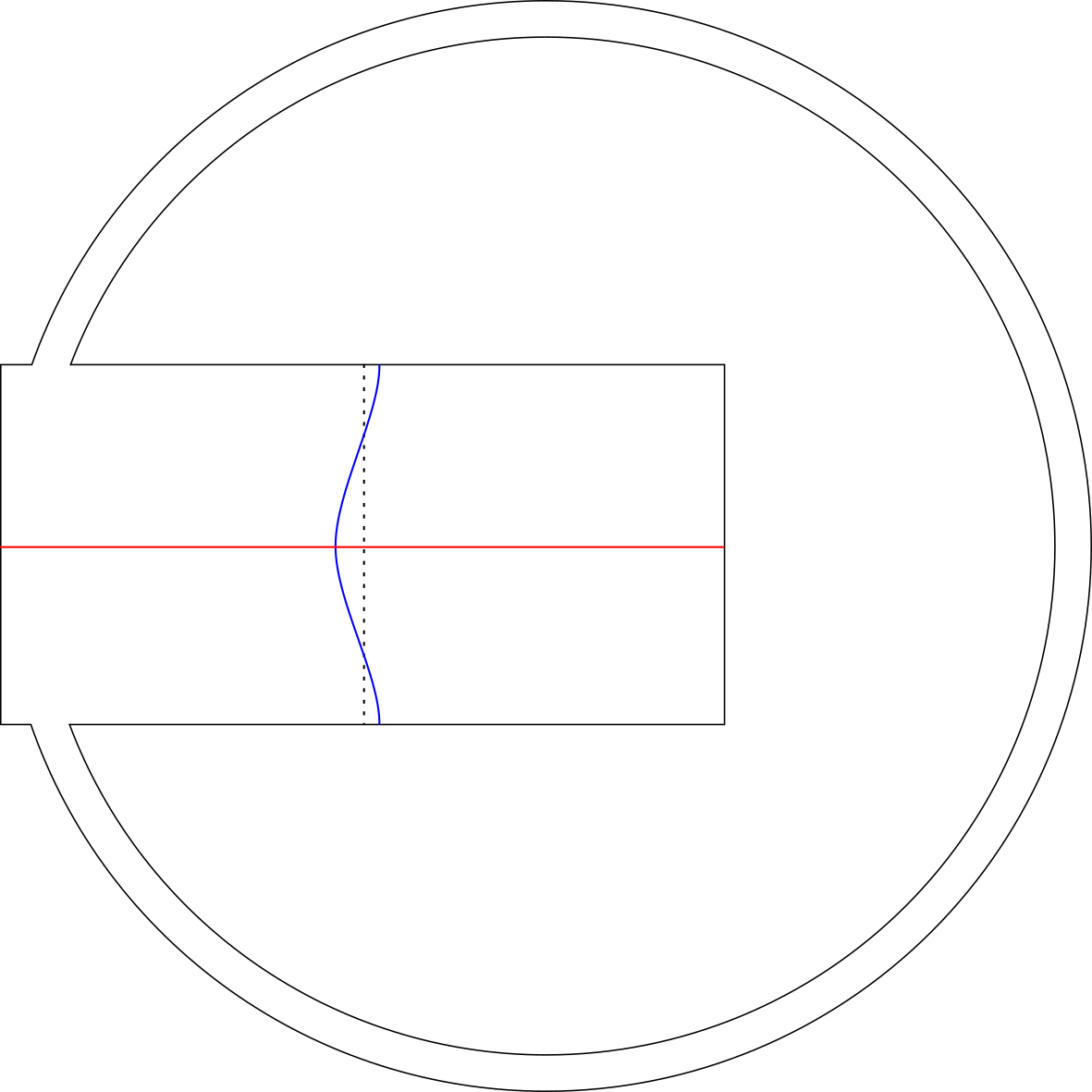}
\quad
\includegraphics[height=.3\textwidth]
{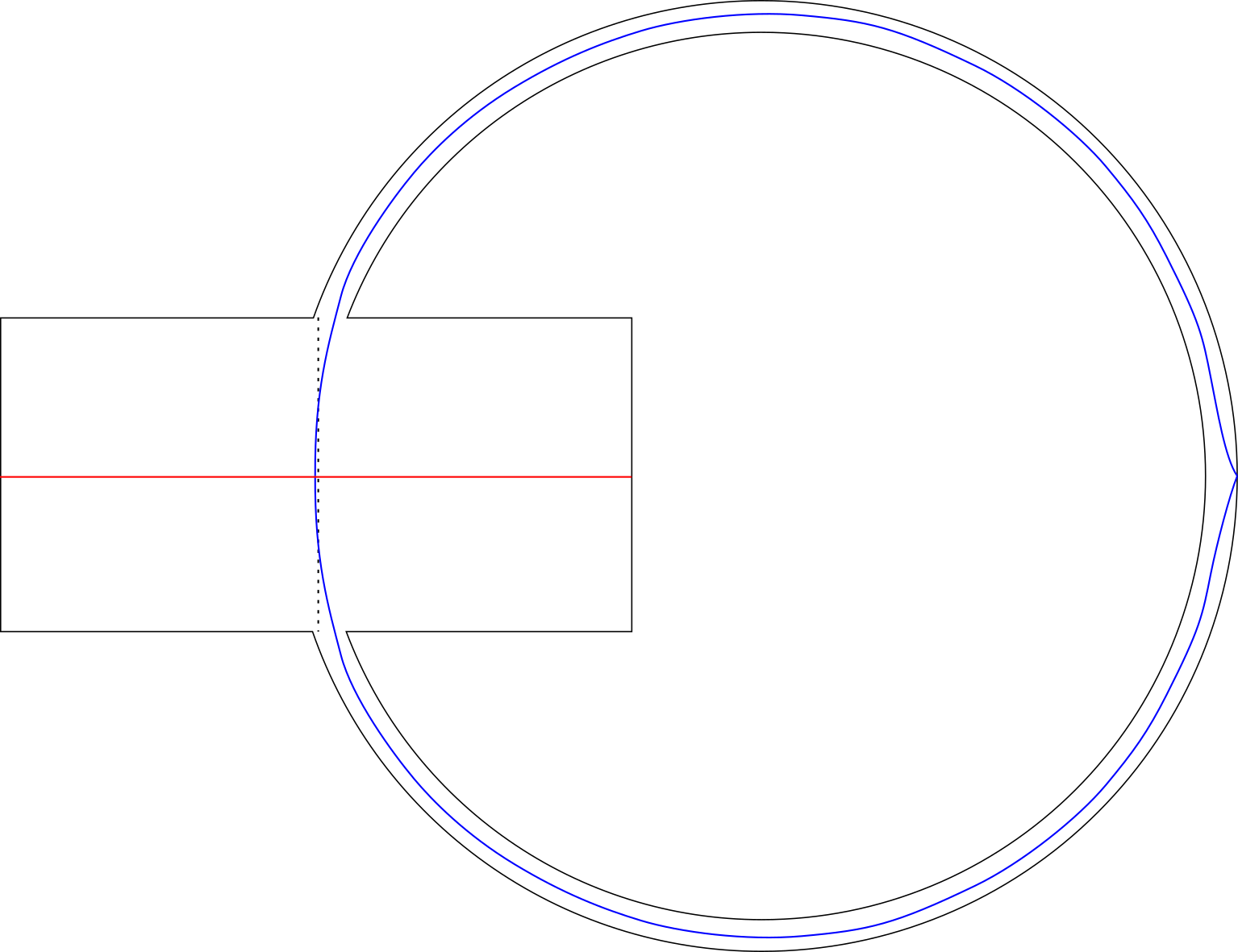}
\caption{Two samples of an \fdsh. The expected nodal line is in blue. The segment $\sigma$ along which the eigenfunction changes sign is in red.}
\end{figure}

The main step towards Theorem~\ref{thm:resultat principal} is to prove that, for an \fdsh, one of the shapes must have a closed nodal line inside the domain. Of course, the definition was tailored to that end.

\begin{theoreme}\label{thm:ligne nodale famille doublement connexe}
Let $(\Omega_t)_{t\in T}$ be an \fdsh. Then, there exists $t_0\in T$ such that the nodal line of a second eigenfunction in $\Omega_{t_0}$ does not touch the boundary, and hence is a closed Jordan curve inside $\Omega_{t_0}$.
\end{theoreme}

\begin{remarque}\label{remarque:localisation de la ligne nodale}
If the \fdsh\/ has the further property that the segment $\sigma$ is a connected component of $\Omega_t\cap\{y=0\}$ for all $t\in T$, and that the second eigenfunction on $\Omega_t$ changes sign an odd number of times along $\sigma$, then the nodal line over $\Omega_{t_0}$ must enclose the hole of the domain. This is because in this case the two boundary components of $\Omega_{t_0}$ cannot belong to the closure of the same nodal domain, hence the nodal line encloses one of those boundary components.
\end{remarque}

The last step is to actually construct an \fdsh. This is ensured by the next proposition.

\begin{proposition}\label{prop:definition de la famille de domaines}
There exists an \fdsh.
\end{proposition}

\begin{remarque}\label{remarque:sign le long du segment}
An inspection of the proof of the proposition shows that the \fdsh\/ constructed (in Section~\ref{sec:existence famille doublement connexe}) satisfies the property described in Remark~\ref{remarque:localisation de la ligne nodale}. This follows from the fact that the second eigenfunction of the domain $\Omega$ of Lemma~\ref{lemme:famille à deux paramètres} is odd with respect to the $y$-axis.
\end{remarque}

Equipped with the previous two results, the proof of Theorem~\ref{thm:resultat principal} follows immediately.

\begin{proof}[Proof of Theorem~\ref{thm:resultat principal}]
Combine Proposition~\ref{prop:definition de la famille de domaines} and Theorem~\ref{thm:ligne nodale famille doublement connexe}.
\end{proof}

The proof of Theorem~\ref{thm:ligne nodale famille doublement connexe} is provided in section~\ref{sec:ligne nodale famille doublement connexe}. The proof of Proposition~\ref{prop:definition de la famille de domaines}, which is actually the most technical part, is postponed to section~\ref{sec:existence famille doublement connexe}.

\section{The nodal line in a family of domains with a sliding handle}\label{sec:ligne nodale famille doublement connexe}

In this section, we prove Theorem~\ref{thm:ligne nodale famille doublement connexe}. To that end, note that if $(\Omega_t)_{t\in T}$ is an \fdsh, the second eigenfunction $u_t$ is symmetric with respect to the $x$-axis, and since the pull-back $\Phi_t$ preserves the symmetry, $v_t$ is also symmetric with respect to the $x$-axis, together with its nodal line and the corresponding impact points (note that the notions of nodal line, nodal domain, and impact point as defined in the introduction make sense for a general function). This property is one of the key ingredients to prove the following result, which states that there cannot be impact points of $v_t$ simultaneaously on the inner and the outer boundary components of $\dd$.

\begin{lemme}\label{lemme:E recouvrent T}
Let $v\in C^1(\dd)\cap C(\overline{\dd})$ be a function with two nodal domains such that, for all $x\in \dd$, at least one of $v(x)$ or
$\nabla v(x)$ is nonzero. Assume furher that $v$ is symmetric with respect to the $x$-axis, and changes sign along a segment on the 
intersection of $\dd$ and the $x$-axis. Then $v$ has a connected nodal line with at most two impact points, which are either both located
on $\partial B_1$ or both on $\partial  B_2$.
\end{lemme}

\begin{figure}[h]
    \centering
    \begin{subfigure}[b]{0.3\textwidth}
        \centering
        \includegraphics[height=.2\textheight]{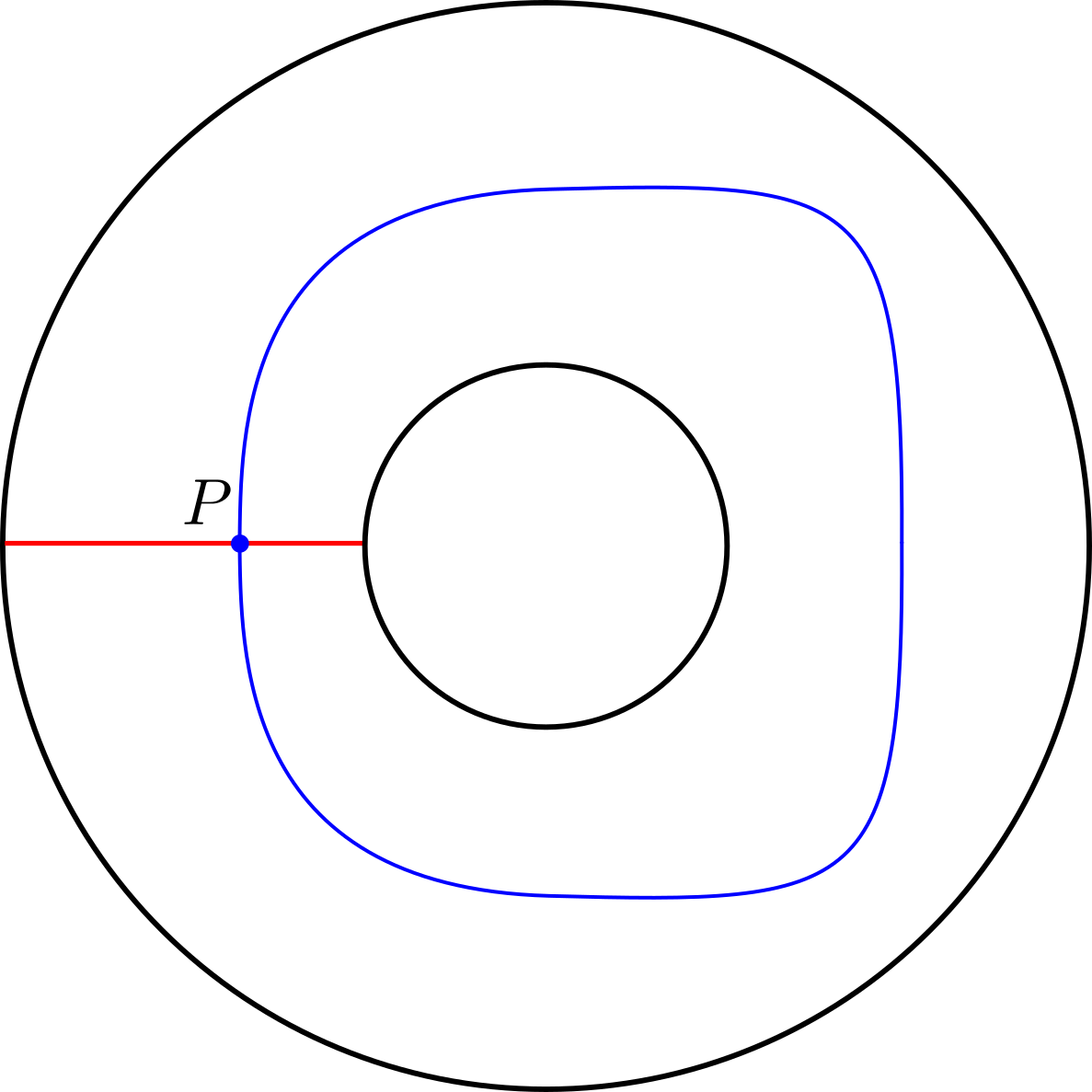}
        \caption{}
        \label{fig:ligne_nodale_anneau_1}
    \end{subfigure}
    \hspace{1em}
    \begin{subfigure}[b]{.3\textwidth}  
        \centering 
        \includegraphics[height=.2\textheight]{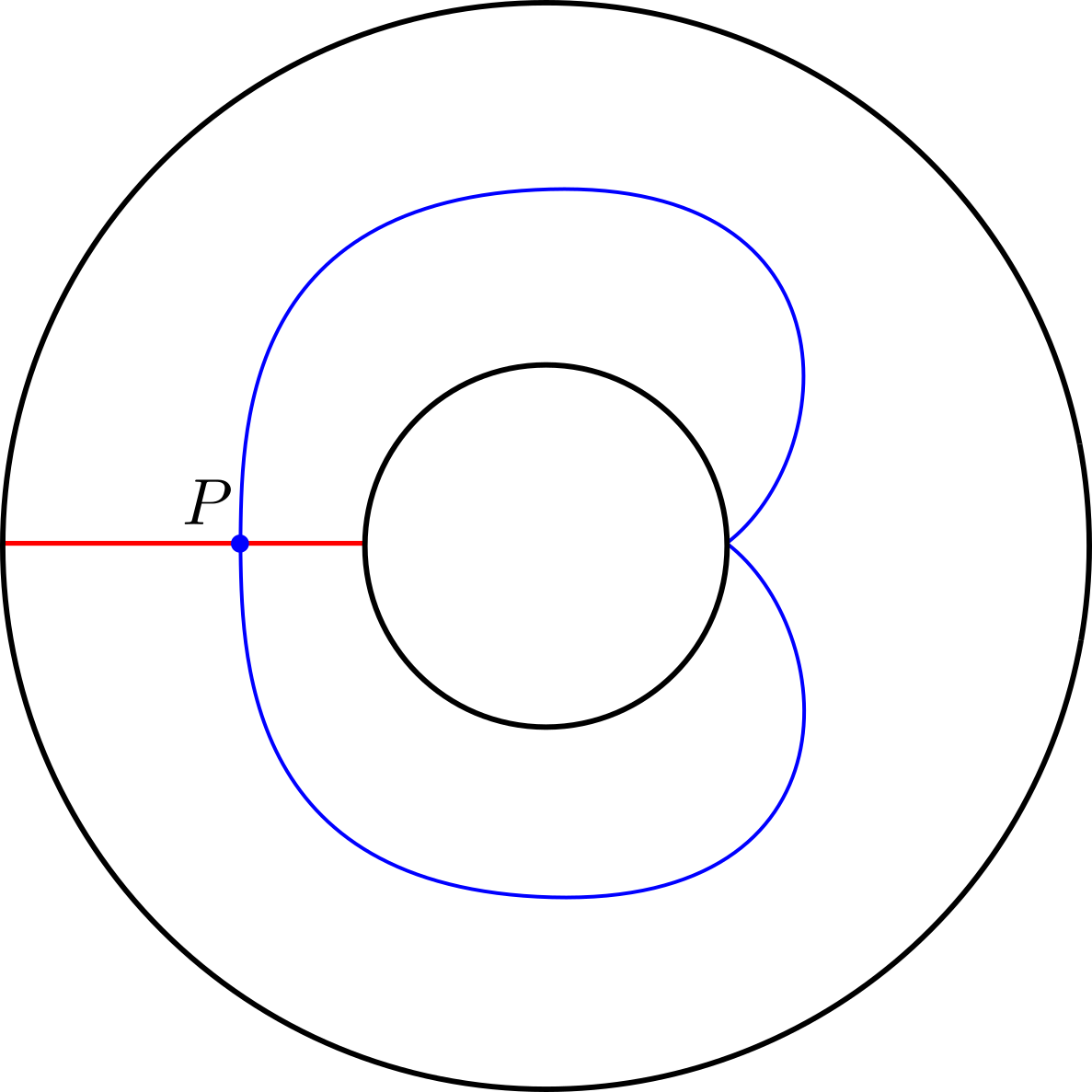}
        \caption{}
        \label{fig:ligne_nodale_anneau_2}
    \end{subfigure}
    \hspace{1em}
    \begin{subfigure}[b]{0.3\textwidth}   
        \centering 
        \includegraphics[height=.2\textheight]{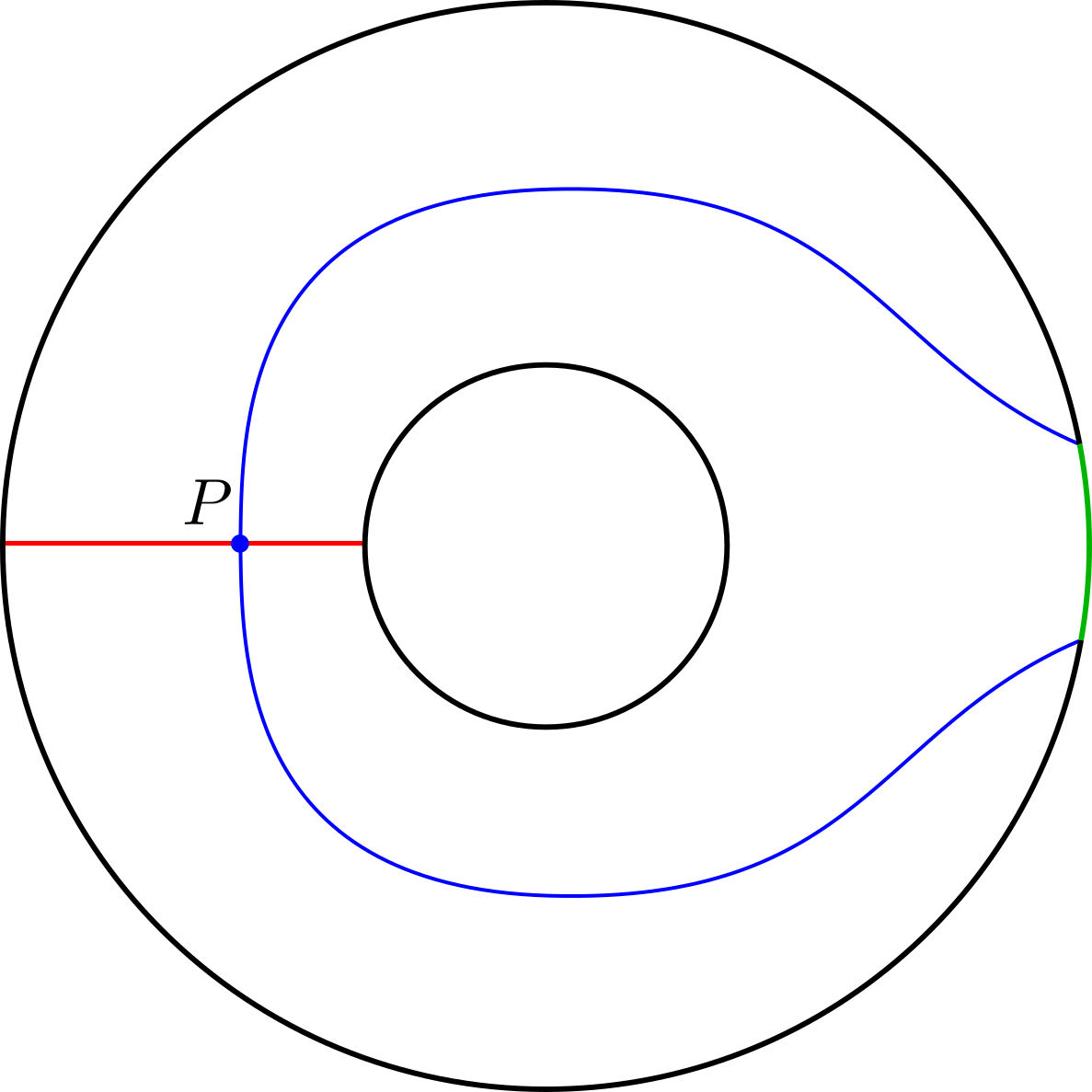}
        \caption{}
        \label{fig:ligne_nodale_anneau_3}
    \end{subfigure}
    \caption{Three typical situations for the function $v$ of Lemma~\ref{lemme:E recouvrent T}. The nodal line is drawn in blue, and $v$ changes sign along the red segment, at the point $P$. The green arc in (\subref{fig:ligne_nodale_anneau_3}) corresponds to the path $\gamma'$ constructed in the proof of the lemma.}
    \label{fig:ligne_nodale_anneau}
\end{figure}

\begin{proof}
By the intermediate value theorem $v$ must vanish at some point $P$ of the segment. We then denote by $\gamma$ the connected component of 
$\mathcal{N}\cap \dd$ that passes through this point, where $\mathcal{N}$ is the nodal line of $v$. Since $v$ has no double zero inside $\dd$, $\gamma$ is a bounded 
connected curve of $\R^2$, by the implicit function Theorem, and it is homeomorphic to either a circle (Figure~\ref{fig:ligne_nodale_anneau_1}) or an open segment whose endpoints lie on $\partial\mathcal{A}$ (Figure~\ref{fig:ligne_nodale_anneau_2} and \ref{fig:ligne_nodale_anneau_3}).

In the first case, by Jordan's Theorem, $\gamma$ splits $\dd$ into two nodal domains. Hence any other portion of the nodal line that would not already be contained in $\gamma$ would yield a third nodal domain, which is impossible. Thus we conclude that $\mathcal{N}=\overline{\gamma}=\gamma\subseteq \dd$. In that situation, there is no impact point at all, and the claim follows.

Assume now that $\gamma$ is homeomorphic to a segment. We point out that $\gamma$ cannot itself be a horizontal segment, otherwise it would contain the segment where $v$ changes sign. Besides, observe that $\gamma$ is symmetric with respect to the $x$-axis. This is because its reflection $\tilde{\gamma}$ also contains the point $P$, hence $\gamma$ and $\tilde{\gamma}$ intersect. As a result, $\gamma\cup\tilde{\gamma}$ is a connected portion of the nodal line and it contains the connected component $\gamma$, hence $\gamma=\gamma\cup\tilde{\gamma}$ is symmetric. Consequently, the boundary of $\gamma$ is made of one point or two symmetric points, which are impact points, both belonging either to $\partial B_1$ or to $\partial B_2$. Since $\partial B_1$ and $\partial B_2$ are arcwise connected, there exists a continuous path $\gamma'$ (drawn in green in Figure~\ref{fig:ligne_nodale_anneau_3}) linking the two impact points and remaining in $\partial \dd$. Note that $\gamma'$ might be reduced to a single point (as illustrated in Figure~\ref{fig:ligne_nodale_anneau_2}). The concatenation of $\gamma$ and $\gamma'$ gives rise to a Jordan curve. By Jordan's Theorem, once again, it means that $\gamma$ splits $\dd$ into two nodal domains. Therefore, we have $\mathcal{N}=\overline{\gamma}$ as before and the only impact points are those determined previously, concluding the proof.
\end{proof}

To apply the previous result, one needs to make sure that $v_t$ admits no double zero in $\dd$, or equivalently that $u_t$ admits no double zero in $\Omega_t$. Since $u_t$ is a second eigenfunction, this follows from Courant's Theorem and from the rather basic observation that one cannot have a self intersection of the nodal line and yet keep the number of nodal domains less than three. This was already remarked in \cite[p.~633 (ii)]{H2ON} (see also \cite[Remark~2.2]{linqun}), and it is actually a consequence of \cite[Proposition~4.1]{helffer-hoffmann-ostenhof-jauberteau-lena}.

The last ingredient of the proof is to establish some kind of continuity for the nodal line of $v_t$ with respect to the deformation parameter $t$. Note that, in general, one cannot expect much continuity on the zero set of a family of functions depending continuously on a parameter (see however \cite{camilli}), since some zeros may appear and disappear unexpectedly. This is the case even with simple one-dimensional examples such as $f_t(x)=x^2-t$. In our particular situation where we already know, by Lemma~\ref{lemme:E recouvrent T}, that the zero set of $v_t$ is very simple and that portions of the nodal line cannot appear and disappear in a complicated way, we shall expect a continuity of the nodal line with respect to $t$ in terms of the Hausdorff distance (see e.g. \cite[section~2.2.3]{henrot-pierre} for an introduction to this distance). Yet, what is needed to complete the proof is actually much weaker, and this is why we only state the following elementary result.

\begin{lemme}\label{lemme:continuité ligne nodale}
Let $\Sigma$ be a subset of $\R^2$ and $v_t\in C(\overline{\dd})$ be a family of functions which is continuous with respect to $t\in T$ uniformly with respect to $x\in \overline{\dd}$. Then, the distance between the nodal set of $v_t$ and $\Sigma$ is lower semi-continuous. More precisely, the function
$$
T\ni t\mapsto \dist(\mathcal{N}(v_t|_{\dd}),\Sigma)
$$
is lower semi-continuous, where \emph{dist} denotes the usual distance between subsets of a metric space, defined as the infimum of the distance between pairs of points in the cartesian product of the two subsets.
\end{lemme}

\begin{proof}
For all $t\in T$, the nodal set $\mathcal{N}(v_t|_{\dd})$ is a compact subset of the compact set $\overline{\dd}$. Therefore, if $t_n\to t$, up to a subsequence, $\mathcal{N}(v_{t_n}|_{\dd})$ converges in the sense of the Hausdorff distance to some compact subset $\mathcal{N}_t$ of $\overline{\dd}$ \cite[Theorem~2.2.23]{henrot-pierre}. Furthermore, $\mathcal{N}_t$ must be a subset of $\mathcal{N}(v_t|_{\dd})$ by the uniform convergence of $v_{t_n}$ to $v_t$. As a result $\dist(\mathcal{N}(v_t|_{\dd}),\Sigma)\leq \dist(\mathcal{N}_t,\Sigma)$, and since $\mathcal{N}(v_{t_n}|_{\dd})$ converges to $\mathcal{N}_t$ we have also $\dist(\mathcal{N}_t,\Sigma)=\lim_n \dist(\mathcal{N}(v_{t_n}|_{\dd}),\Sigma)$ \cite[Proposition~2.2.25]{henrot-pierre}. Since the initial subsequence was taken arbitrarily, we obtain $\dist(\mathcal{N}(v_t|_{\dd}),\Sigma)\leq\liminf_n\dist(\mathcal{N}(v_{t_n}|_{\dd}),\Sigma)$, concluding the proof.
\end{proof}

We are now ready to complete the proof of Theorem~\ref{thm:ligne nodale famille doublement connexe}.

\begin{proof}[Proof of Theorem~\ref{thm:ligne nodale famille doublement connexe}]
We consider an \fdsh\/ $(\Omega_t)_{t\in T}$ and pull back the eigenfunctions $u_t$ into the fixed annulus $\dd = B_2\setminus \overline{B_1}$ thanks to the map $\Phi_t:\dd\to\Omega_t$. The
pulled-back function $v_t=u_t\circ \Phi_t$ has two nodal domains, is symmetric with respect to the $x$-axis, and changes sign along the segment $\Phi_t^{-1}(\sigma)$, which is contained in the $x$-axis since $\Phi_t$ preserves the symmetry. Observe also that $u_t$ cannot have a double zero inside $\Omega_t$ by \cite[Proposition~4.1]{helffer-hoffmann-ostenhof-jauberteau-lena}, thus the same holds for $v_t$ in $\dd$. Applying Lemma~\ref{lemme:E recouvrent T}, we conclude that the nodal line $\mathcal{N}_t$ of $v_t$ does not touch $\partial B_1$ and $\partial B_2$ simultaneaously. Furthermore, we know that $\mathcal{N}_{t_-}$ does not intersect $\partial B_2$ while $\mathcal{N}_{t_+}$ does not intersect $\partial B_1$. Assume for the purpose of a contradiction that the nodal line of $v_t$ always touches the boundary of $\Omega_t$, so that the distance between $\mathcal{N}_{t}$ and $\partial \dd$ is always zero. Define $\mathcal{I}_k$ to be the set of times $t\in T$ such that the nodal line of $v_t$ touches $\partial B_k$, that is
$$
\mathcal{I}_k:=\{t\in T: \dist(\mathcal{N}_{t},\partial B_k) = 0\}.
$$
By the previous discussion, we know that the set $\mathcal{I}_1$ contains $t_-$, the set $\mathcal{I}_2$ contains $t_+$, and those two sets form a partition of the interval $T$. Lastly, being closed lower-level sets of a lower-semi continuous function by Lemma~\ref{lemme:continuité ligne nodale}, $\mathcal{I}_1$ and $\mathcal{I}_2$ are closed subsets of $T$. This contradicts the connectedness of $T$. Therefore, $\mathcal{N}_{t_0}$ must remain closed inside $\mathcal{A}$ for some value of $t_0\in T$, and similarly the nodal line of $u_{t_0}$ is closed inside $\Omega_{t_0}$, as claimed. Since $u_{t_0}$ has no double zero, its nodal line is a Jordan curve.
\end{proof}

\section{Existence of a family of domains with a sliding handle}\label{sec:existence famille doublement connexe}

\newcommand{\dbprm}{\Theta}

The purpose of this section is to prove Proposition~\ref{prop:definition de la famille de domaines}. The construction of the \fdsh\/ will be performed
following the procedure described in the Introduction. In particular, we begin by constructing a family depending on two parameters, one responsible for sliding and the other for shrinking the annulus-like handle. For the remaining of the proof, this two-parameter family must satisfy several properties that we list in the following result (see also Figure~\ref{fig:famille_reguliere}).

\begin{lemme}\label{lemme:famille à deux paramètres}
Let $W=(-1,0)$ and $E=(1,0)$. There exists a family of bounded connected open subsets $(\dbprm_{h,t})_{0\leq h<h_0,t\in T}$ of $\R^2$, for some $h_0>0$ and for an open interval $T\subseteq \R$, such that:
\begin{enumerate}
\item For all $0\leq h<h_0$ and all $t\in T$, $\dbprm_{h,t}$ is symmetric with respect to the $x$-axis.
\item For all $t,s\in T$, $\dbprm_{0,t}=\dbprm_{0,s}=:\Omega$.
\item The set $\Omega$ is symmetric with respect to the $y$-axis, the second eigenvalue of $\Omega$ is simple and the second eigenfunction is odd with respect to the $y$-axis. 
\item For all non-zero $0<h<h_0$, $\dbprm_{h,t}$ is doubly-connected. One connected component of its boundary contains $W$ while the other contains $E$ - we call them the left and the right boundary components, respectively. Furthermore, $\dbprm_{h,t}$ contains the whole open line segment $(W,E)$, and there is a neighbourhood $\mathcal{U}$ of this segment, independant of $h$ and $t$, such that $\dbprm_{h,t}\cap\mathcal{U}=\Omega\cap\mathcal{U}$.
\item\label{it:t+ et t-} There exist $t_-,t_+\in T$ such that the vertical segment bisecting $\Omega$ does not touch the left (resp. right) boundary component of $\dbprm_{h,t}$ for $t=t_+$ (resp. $t=t_-$) and for all $0\leq h<h_0$. Furthermore, there exists a neighbourhood $\mathcal{V}$ around the $y$-axis such that $\Omega\cap\mathcal{V}$ is a connected component of $\dbprm_{t_\pm,h}\cap\mathcal{V}$, for all $0\leq h<h_0$.
\item The sets $\dbprm_{h,t}$ are continuous with respect to $(h,t)\in [0,h_0[\times T$, in the sense of Mosco convergence.
\item For each $0<h<h_0$, the sets $(\dbprm_{h,t})_{t\in T}$ are smooth with a local $C^1$ dependence on $t\in T$. This means that in a neighbourhood of each point of $T$, there exists a family of $C^\infty$-diffeomorphisms $\Phi_{h,t}$ mapping the annulus $\dd$ to $\dbprm_{h,t}$, such that $(t\mapsto\Phi_{h,t},\Phi_{h,t}^{-1}\in C^\infty(\R^d,\R^d)^2)$ is $C^1$. Furthermore, we shall take the transformation $\Phi_{h,t}$ preserving the symmetry with respect to the $x$-axis, and such that $\Phi_{h,t}^{-1}(W)$ belongs to the outer boundary component and $\Phi_{h,t}^{-1}(E)$ belongs to the inner boundary component of $\dd$, for all $t\in T$.
\end{enumerate}
\end{lemme}

\begin{figure}[h]
\begin{subfigure}{.45\textwidth}
\includegraphics[height=.2\textheight]
{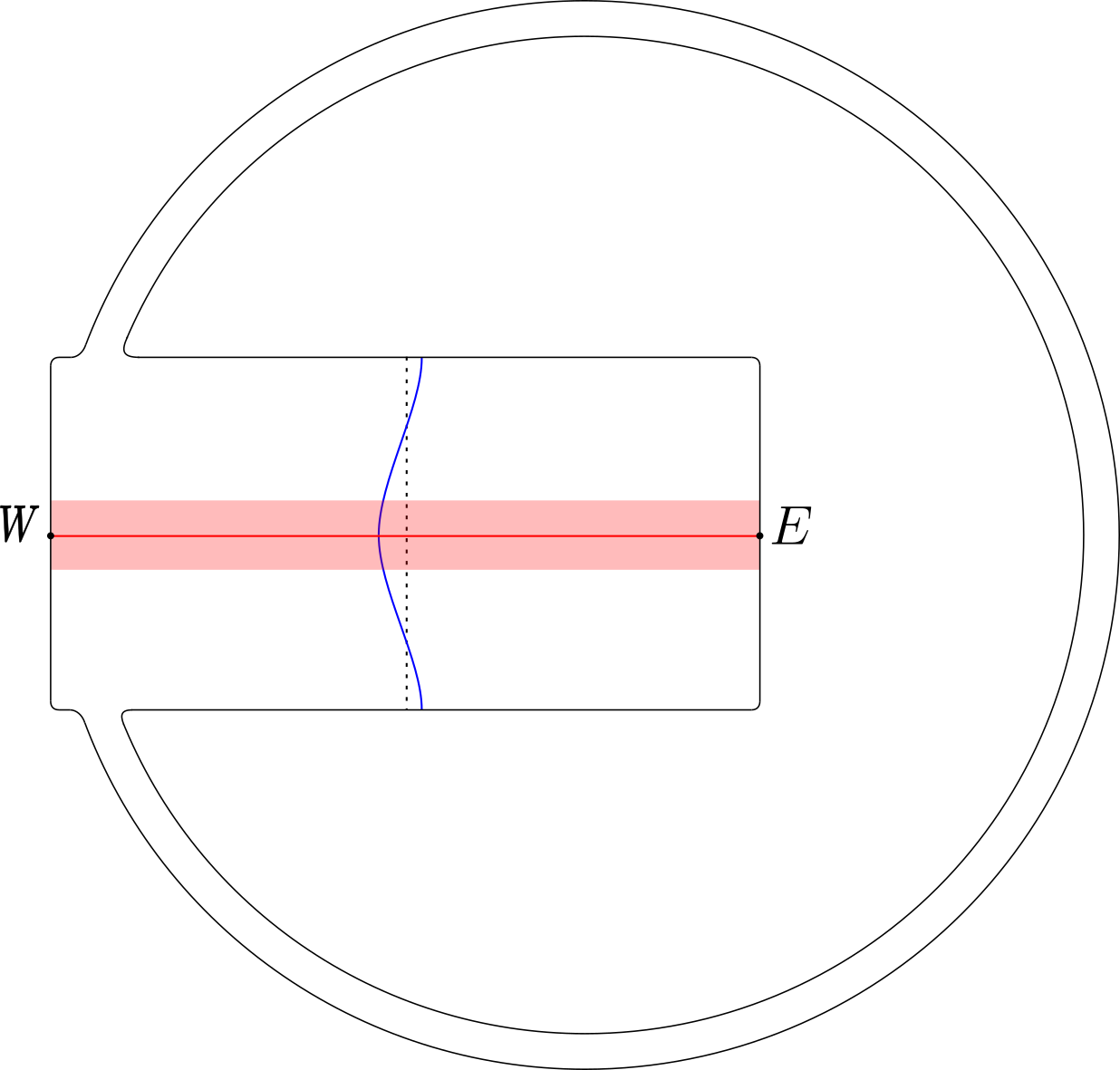}
\caption{$t_-$}
\end{subfigure}
\quad
\begin{subfigure}{.45\textwidth}
\includegraphics[height=.2\textheight]{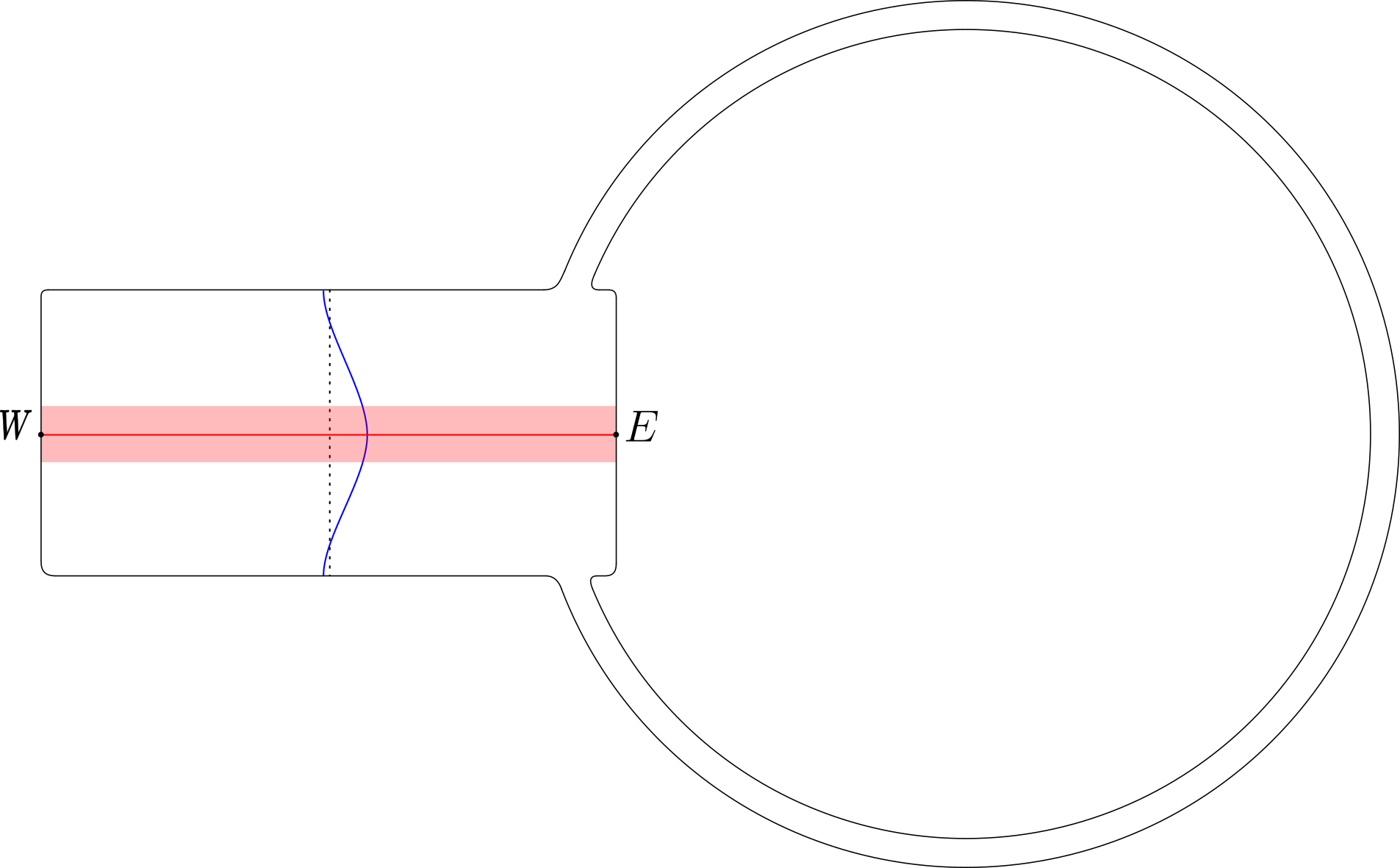}
\caption{$t_+$}
\end{subfigure}
\caption{The domain $\dbprm_{h,t}$ constructed in Lemma~\ref{lemme:famille à deux paramètres} at $t=t_-$ and $t=t_+$, for some $0<h<h_0$.
The segment $(W,E)$ is drawn in dark red, while its neighbourhood $\mathcal{U}$ is displayed in light red.}
\label{fig:famille_reguliere}
\end{figure}

We refer to \cite{daners} for an introduction to Mosco convergence, and only recall that this convergence guarantees that of Dirichlet eigenvalues and eigenfunctions (or more precisely of the spectral projectors). The first six properties above are satisfied by the family of domains described in the introduction. The last one, which is an ancillary 
smoothness property will be needed to prove that a second eigenfunction of the \fdsh\/ to be constructed, when pulled-back in $\dd$, is continuous in $t$ uniformly with respect to $x\in\overline{\dd}$ (see the appendix). To ensure this regularity property, we will round up the corners of the domains in the family (see Figure~\ref{fig:famille_reguliere}).

\begin{proof}
Following the construction proposed in the introduction, we begin with the horizontally elongated rectangle
$$
R=(-1,1)\times(-l,l),\qquad 0<l<1,
$$
and introduce the following smooth perturbation of $R$. Let $\Omega$ be the domain obtained from $R$ by rounding the corners, in such a way that $\Omega$ is convex, smooth, symmetric with respect to both axes, and coincides with $R$ except in some small neighbourhood of its corners. If the neighbourhood is small enough, then the second eigenvalue of $\Omega$ is still simple and the corresponding eigenfunction is odd with respect to the vertical axis. We also require that the neighbourhood is so small that the boundary of $\Omega$ still coincides with a horizontal segment $\Sigma_N:=(P_N,Q_N)$ near the point $N:=(0,l)$ and a horizontal segment $\Sigma_S:=(P_S,Q_S)$ near the point $S=(0,-l)$.

Then, consider a large circle of radius $r$ that contains $\Omega$ compactly, and enlarge the circle to obtain an open annulus $A_h$ of small positive width $h$. We also write $A_0=\emptyset$. When $h$ is small enough, say $0<h<h_0<l$, the annulus still encloses $\Omega$. Translate the annulus horizontally, in such a way that the two connected components of the boundary of the translated annulus intersect $\Sigma_N$ and $\Sigma_S$. There exists an open interval $T\subseteq\R_+$ such that the two boundary components of the translated annulus $A_h(t):=A_h+(t,0)$ intersect those two segments when $t\in T$. As a result of this construction, for all $t\in T$ the boundary of $A_h(t)\cup \Omega$ is made of two connected components. We call the boundary component containing $W$ the left boundary component, and that containing $E$ the right boundary component. We shall assume that $T$ is taken large enough and $h_0$ small enough so that for $t\in T$ and $0<h<h_0$, the closed line segment $[S,N]$, which is the nodal line of the eigenfunction on $\Omega$, intersects only the right boundary component when $t=t_-$, for some $t_-\in T$ close to $\inf T$, and only the left boundary component when $t=t_+$, for some $t_+\in T$ close to $\sup T$. Up to shrinking slightly $h_0$ we have that, in a neighbourhood of this segment, $\Omega$ coincides with $\Omega\cup A_h(t_-)$ and with $\Omega\cup A_h(t_+)$, for all $0\leq h< h_0$. Similarly, in a fixed neighbourhood of the segment $[W,E]$, $A_h(t)\cup \Omega$ coincides with $\Omega$, for all $t\in T$ and all $0\leq h<h_0$. Lastly, by standard results, e.g. \cite[Theorem~7.5]{daners}, $A_h(t)\cup \Omega$ is continuous with respect to $(h,t)$ in the sense of Mosco convergence, even when $h=0$. Indeed, the first condition in \cite[Theorem~7.5]{daners} follows from the fact that $A_h(t)\cup \Omega$ is continuous in the sense of inner compact sets, while the second one is a consequence of the continuity in the sense of characteristic functions and of the Faber-Krahn inequality (see \cite[chapter~2]{henrot-pierre} for an introduction to the different types of convergence for domains of $\R^n$). The last condition is also satisfied since $A_h(t)\cup \Omega$, being Lipschitz regular, is stable in the sense that $H_0^1(A_h(t)\cup \Omega)=H_0^1(\overline{A_h(t)\cup \Omega})$ (see \cite[p.~608]{daners}).

However, $A_h(t)\cup \Omega$ is not regular at the four intersection points between the boundaries of $A_h(t)$ and $\Omega$. Yet we shall replace the family $A_h(t)\cup \Omega$ by a family of smooth domains $\dbprm_{h,t}$ with $C^1$ dependence on $t$ in the following way. Let $p$ be one of the two irregular points of the left boundary component of $A_h(t)\cup \Omega$. Up to shrinking
slightly the interval $T$ (but not too much, to ensure $t_-,t_+\in T$) and the threshold $h_0$, we shall assume that the distance between $p$ and any of $P_N,P_S$ is not smaller than $h$. In this way, the intersection between the ball $B_{h/2}(p)$ and the boundary of $A_h(t)\cup \Omega$ is, up to reflections and translations the graph of the function
$$
f_h:(-r-h,-r)\ni x \mapsto
\begin{cases}
0 & \text{if }x< -r-h/2,\\
\sqrt{(r+h/2)^2-x^2} &\text{if }x\geq -r-h/2.
\end{cases}
$$
We then consider the function $g_{h}=f_{h}(1-\chi_h)$ where $\chi_h$ is a smooth cutoff function with support in $(-r-h,-r)$ and such that $\chi_h=1$ around $-r-h/2$. To construct $\dbprm_{h,t}$, we replace the function $f_{h}$ by the function $g_{h}$. Clearly we can ensure that the resulting $\dbprm_{h,t}$ is symmetric with respect to the $x$-axis. Note that with this construction the new boundary around $p$ remains invariant with respect to $t$, up to a translation. A similar procedure allows to smoothen the boundary near the two irregular points $p$ of the right boundary component of $A_h(t)\cup \Omega$, in the ball $B_{h/2}(p)$. The set $\dbprm_{h,t}$ obtained in this way is now smooth and still has two boundary components. Observe also that $\dbprm_{h,t}$ and $A(h)\cup \Omega$ differ only inside the ball $B_{h/2}(p)$. Hence we still have that $W$ and $E$ belong to distinct boundary components of $\dbprm_{h,t}$.
Furthermore, the segment in between these two points is contained in $\dbprm_{h,t}$, and $\dbprm_{h,t}$ always coincides with $\Omega$ in a fixed neighbourhood $\mathcal{U}$ of that segment. Lastly, for $t=t_-$ the vertical segment $[S,N]$ intersects only the right boundary component of $\dbprm_{h,t}$, up to taking a smaller value of $h_0$.
When $t=t_+$, it intersects only the left boundary component. In both cases, $\dbprm_{t_\pm,h}$ still coincides with $\Omega$ in a neighbourhood of that segment.

In terms of continuity, since $\dbprm_{h,t}$ is smooth it is stable, and by \cite[Theorem~7.5]{daners} once again, $(h,t)\mapsto\dbprm_{h,t}$ is continuous in the sense of Mosco convergence. To conclude we need to verify the last point of the lemma. To that end, we observe that apart from two vertical edges, the boundaries of $\dbprm_{h,t}$ are translations of each other as $t$ varies. Hence for a fixed $0<h<h_0$, the domains $\dbprm_{h,t}$ can be obtained from each other through a family of $C^\infty$-diffeomorphisms depending smoothly on $t$ in any small enough open subinterval $I$ of $T$. Indeed, for $t\in T$ take a transformation of the form $id+(s-t,0)\phi_t$ where $\phi_t:\R^2\to\R$ is supported in the annulus of $A_{3h}(t)$ and is equal to one in the annulus $A_{2h}(t)$. Note that this map just translates horizontally the points near the annulus and leaves invariant the other points. Whenever $s\in T$ is close enough to $t$, this transformation is a $C^\infty$-diffeomorphism depending smoothly on $s$, preserving the symmetry with respect to the $x$-axis, and mapping $\dbprm_{h,t}$ to $\dbprm_{h,s}$. Then, fixing $t_0\in I$, since $\dbprm_{h,t_0}$ is smooth and doubly-connected there exists a $C^\infty$-diffeomorphism mapping $\dd$ to $\dbprm_{h,t_0}$, thus composing this transformation with the previous $\dbprm_{t_0}\to\Omega_t$ for $t\in I$, we end up with a family of $C^\infty$-diffeomorphisms $\Phi_{h,t}:\dd\to\dbprm_{h,t}$ depending smoothly on $t\in I$.
\end{proof}

The Mosco convergence of the domains constructed in Lemma~\ref{lemme:famille à deux paramètres}, combined with the simplicity of the second eigenvalue over $\Omega$, allows to prove some continuity of the eigenfunctions with respect to $(h,t)$ jointly. This is crucial in order to argue that when the parameter $h$ takes on a small enough fixed value, the eigenfunction on $\dbprm_{h,t}$ still changes sign on the segment $(W,E)$ for all $t\in T$. This allows us to construct a one-parameter \fdsh, and hence to prove Proposition~\ref{prop:definition de la famille de domaines}.

\begin{proof}[Proof of Proposition~\ref{prop:definition de la famille de domaines}]
Consider the two-parameter family $\dbprm_{h,t}$ provided by Lemma~\ref{lemme:famille à deux paramètres} where $0\leq h<h_0$ and $t\in T_0$. We first prove that for some fixed small enough $h_1<h_0$, the second eigenvalue of $\dbprm_{h,t}$ is simple and the associated eigenfunctions change sign along the segment $(W,E)$, for all $t$ in an arbitrary compact subset of $T_0$, and for all $0\leq h<h_1$. By contradiction, otherwise we would have a sequence $t_n\to t$ for some $t\in T_0$ and a sequence $h_n\to0$, such that either the second eigenvalue is not simple or the associated eigenfunctions have fixed sign on $(W,E)$. But by Mosco continuity, when $n\to\infty$ the eigenvalues over $\dbprm_{h_n,t_n}$ converge to those
over $\dbprm_{0,t}=\Omega$, and the eigenprojectors converge strongly as operators of $L^2(\R^d)$ \cite[Theorem~4.4 and Corollary~4.2]{daners}. Since $\Omega$ has a simple second eigenvalue, this means that the second eigenvalue on $\dbprm_{h_n,t_n}$ is also simple for large enough $n$. Furthermore,  local elliptic estimates, e.g. \cite[Theorem 2.20]{gazzola-grunau-sweers}, show that the corresponding eigenfunctions converge not only in $L^2(\R^d)$, but also in any Sobolev space $H^k(\Omega\cap\mathcal{U})$, since $\partial\Omega$ is smooth, and in particular in $C(\overline{\Omega\cap\mathcal{U}})$ by Sobolev embeddings. Here, we used that $\dbprm_{h,t}\cap\mathcal{U}=\Omega\cap\mathcal{U}$. As a consequence, the limit eigenfunction has fixed sign along $(W,E)$, which is a contradiction. Actually, since the limit eigenfunction is odd with respect to the $y$-axis we shall even conclude that the eigenfunction on $\Theta_{h,t}$ changes sign an odd number of time along $(W,E)$ for $h<h_1$, which justifies Remark~\ref{remarque:sign le long du segment}.

The purpose of the next step is to prove that for $t=t_\pm$ and $h$ sufficiently small, the nodal line of the second eigenfunction in $\dbprm_{h,t_\pm}$ is localised in an arbitrary neighbourhood of the $y$-axis. In view of point~\ref{it:t+ et t-} in Lemma~\ref{lemme:E recouvrent T}, this will show that for $t=t_-$ the nodal line cannot touch the right boundary componen of $\dbprm_{h,t_-}$, while for $t=t_+$ it cannot touch the left boundary component of $\dbprm_{h,t_+}$. To prove this claim, let $\mathcal{V}$ be a neighbourhood of the $y$-axis, containing the nodal line of the second eigenfunction on $\Omega$. Note that we shall take a vertical strip as this neighbourhood (in blue in Figure~\ref{fig:famille_reguliere et voisinage V}), and this choice will make the following discussion clearer. As before, the second eigenvalue on $\dbprm_{h,t_\pm}$ converges to the second eigenvalue on $\Omega$ as $h\to0$, and the corresponding eigenfunctions $u_{h,\pm}$ converge to the second eigenfunction $u$ on $\Omega$ in $L^2(\R^d)$, and then in $C^1(\overline{\Omega\cap\mathcal{V}})$ by local elliptic estimates. To apply the elliptic estimates, we used that $\Omega\cap\mathcal{V}$ is a connected component of $\dbprm_{h,t}\cap\mathcal{V}$, by point~\ref{it:t+ et t-} in Lemma~\ref{lemme:famille à deux paramètres}. As explained below, this shows that if $h_2<h_1$ is small enough, for any $0<h<h_2$ the nodal line of $u_{h,\pm}$ must be trapped in $\mathcal{V}\cap\Omega$. Indeed, for any compact $K$ contained in $\Omega$ (in light green in Figure~\ref{fig:famille_reguliere_et_voisinage_V}), the eigenfunction $u$ does not vanish and it has a different sign on the left and the right vertical edges of $\partial\mathcal{V}\cap K$ (in dark red in Figure~\ref{fig:famille_reguliere_et_voisinage_V}). Therefore the same holds for $u_{h,\pm}$ for small enough $h$. Similarly, the normal derivative of $u$ does not vanish on the compact set $\partial\Omega\cap \partial\mathcal{V}$, hence the same holds for $u_{h,\pm}$ for small enough $h$. Overall, we get that $u_{h,\pm}$ does not vanish and has a different sign on the two vertical edges of $\partial\mathcal{V}\cap \Omega$ when $h<h_2$ with $h_2$ sufficiently small. As a consequence, the nodal line of $u_{h,\pm}$ must be contained in $\mathcal{V}\cap\overline{\Omega}$. Since $\partial\Omega\cap(\dbprm_{h,t_+}\setminus \Omega)$ lies in the right half-plane, this means that the impact points of the nodal line of $u_{h,+}$ are located on the left boundary component of $\dbprm_{h,t_+}$, for all $0<h<h_2$. Analogously, the impact points of nodal line of $u_{h,-}$ are located on the right boundary component of $\dbprm_{h,t_-}$, for $0<h<h_2$.

\begin{figure}[h]
\begin{subfigure}{.45\textwidth}
\includegraphics[height=.2\textheight]
{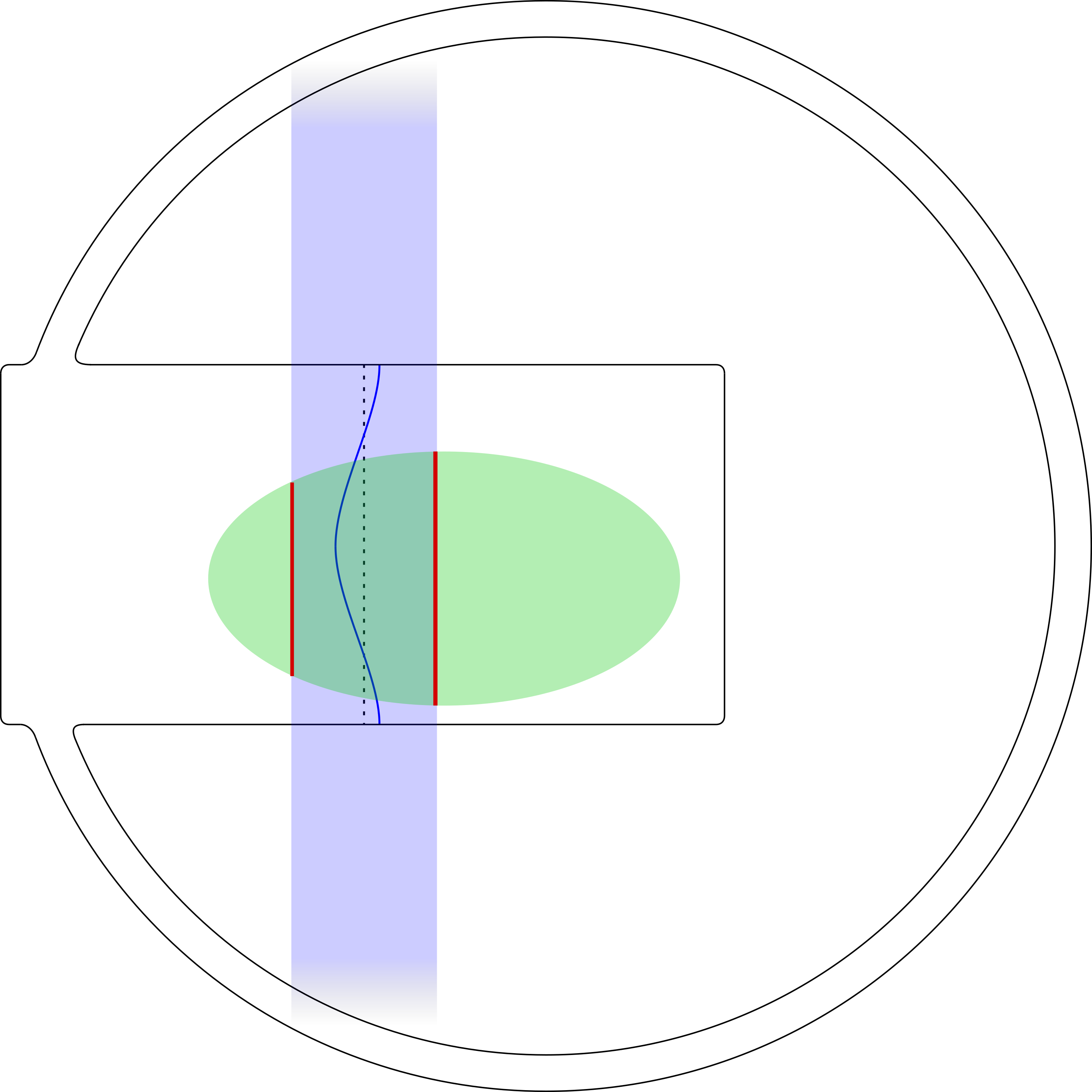}
\caption{$t_-$}
\end{subfigure}
\quad
\begin{subfigure}{.45\textwidth}
\includegraphics[height=.2\textheight]{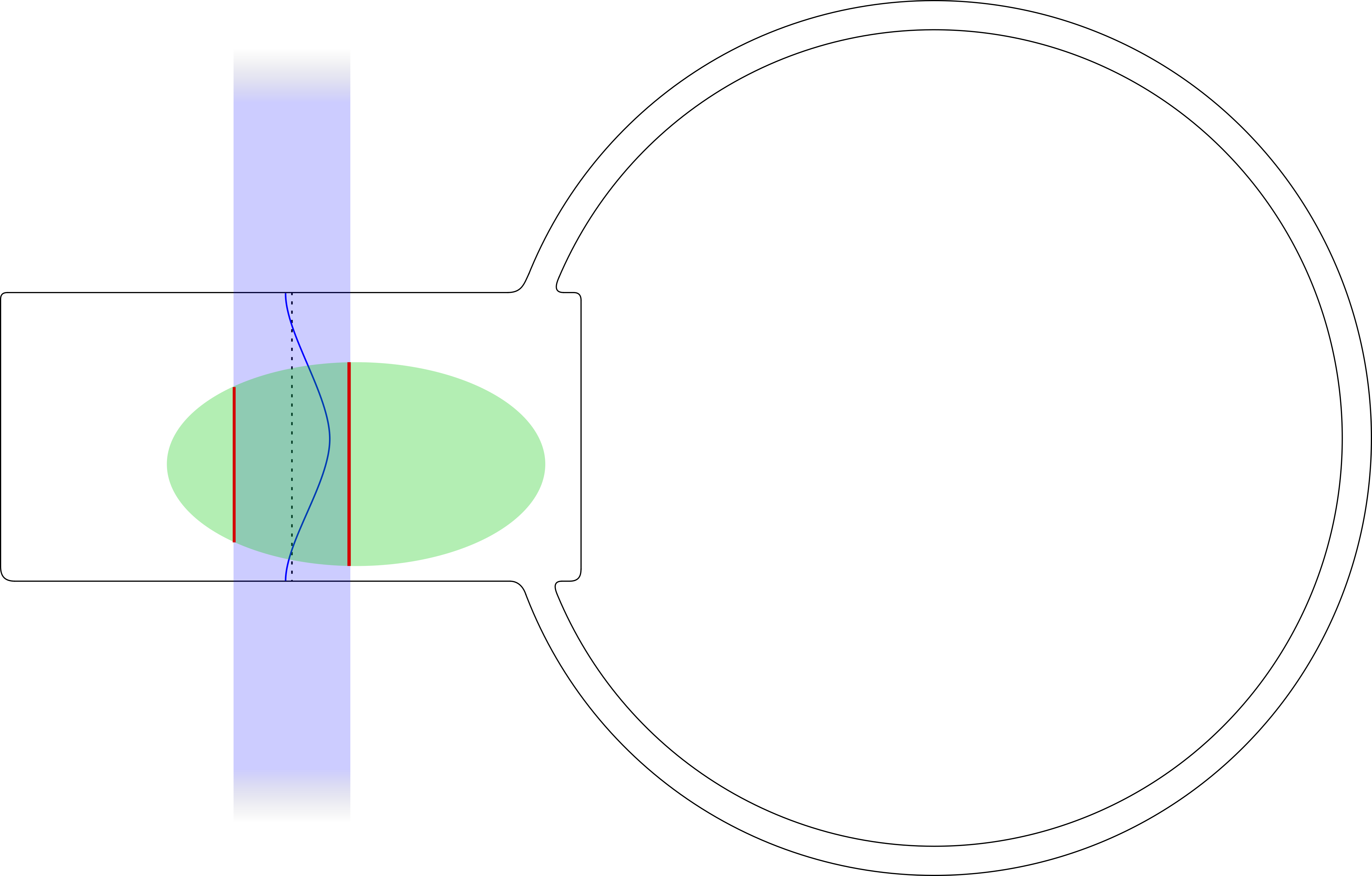}
\caption{$t_+$}
\end{subfigure}
\caption{The domain $\dbprm_{h,t}$ constructed in the proof of Proposition~\ref{prop:definition de la famille de domaines} at $t=t_-$ and $t=t_+$, for some $0<h<h_2$. The vertical strip in light blue represents the neighbourhood $\mathcal{V}$. The green region is an arbitrary compact set $K\subseteq\Omega$.}
\label{fig:famille_reguliere_et_voisinage_V}
\end{figure}

In what follows, we restrict to an open subinterval of $T$ which is compactly contained in $T_0$ but still contains $t_-$ and $t_+$. We then define the single parameter family of domains $\Omega_t:=\dbprm_{t,h_2/2}$, for all $t\in \overline{T}$. In this way, for all $t\in \overline{T}$, the second eigenvalue on $\Omega_{t}$ is simple and the corresponding eigenfunction, that we denote $u_t$, changes sign along the segment $(W,E)$ for all $t\in \overline{T}$. Furthermore, the impact points of the nodal line of $u_{t_+}$ lie on the left boundary component of $\Omega_{t_+}$, while that of $u_{t_-}$ lie on the right boundary component of $\Omega_{t_-}$. We also have the existence locally around each point of $\overline{T}$ of a family of $C^\infty$-diffeomorphisms $\Phi_t:\R^2\to\R^2$ preserving the symmetry with respect to the $x$-axis, mapping $\dd$ to $\Omega_t$, and having a $C^1$ dependence with respect to $t\in \overline{T}$. Up to a geometric inversion, we shall assume that $\Phi_t$ maps the outer boundary component of $\dd$ to the left boundary component of $\Omega_t$ and the inner boundary component of $\dd$ to the right boundary component of $\Omega_t$. Lemma~\ref{lemme:continuité fonction propres par rapport à t} in appendix, which is based on the implicit function Theorem, combined with Sobolev embeddings shows that the pulled back eigenfunctions $v_t:=u_t\circ\Phi_t$ are then continuous with respect to $t$ uniformly with respect to $x\in\overline{\dd}$. This holds locally around each point of $\overline{T}$, but covering this compact interval by a finite number of small enough open intervals and gluing the maps $t\mapsto\Phi_t$ obtained in each of them accordingly, we obtain the desired family of transformations $\Phi_t$ for all $t\in T$.
\end{proof}

%% file: parties/annexe.tex
\section{Continuity of eigenfunctions under domain transformation}

In this appendix, we prove a continuity result for Dirichlet eigenfunctions in regular spaces under domain transformation. As can be seen
in~\cite[section 4]{burenkov-lamberti-lanza_de_cristoforis} (see also~\cite[section 5.7]{henrot-pierre}), this can classically be
obtained thanks to the implicit function theorem. However, in most references the continuity is obtained in spaces with low regularity,
typically in $L^2(\R^2)$. Actually, being defined on different domains, one cannot expect to prove continuity in spaces more regular
than $W^{1,p}(\R^2)$ the eigenfunctions, since they are not more than Lipschitz over $\R^2$ in general. But as long as one pull them back in a fixed domain, one shall obtain much more regularity.

\begin{lemme}\label{lemme:continuité fonction propres par rapport à t}
Let $D$ be a $C^k-$regular bounded open set and $\Phi_t:\R^d\to\R^d$ be a family of $C^k$-diffeomorphisms, for $t$ in an open set $T$, such
that $(T\ni t\mapsto\Phi_t,\Phi_t^{-1}\in C^k(\R^d,\R^d)^2)$ is $C^1$. Define $\Omega_t:=\Phi_t(D)$ and assume that for $t_0\in T$, $\lambda$ is a simple eigenvalue
on $\Omega_{t_0}$. Then, for $t$ in some neighborhood of $t_0$, there exists a family of eigenvalues $\lambda_t$ on $\Omega_t$ and a family
of associated eigenfunctions $u_t$ such that $\lambda_{t_0}=\lambda$ and such that the following function is $C^1$ in the neighborhood of $t_0$:
$$
t\mapsto (\lambda_t,u_t\circ\Phi_t)\in \R\times H^k(D)
$$
\end{lemme}

\begin{proof}
We closely follow \cite[p.210--211]{henrot-pierre} and define the functional $\mathcal{F}:\R\times H_0^1\cap H^k(D)\times\R\to H^{k-2}(D)\times\R$ given by
$$
\mathcal{F}(t,v,\lambda)=\left(-\dv(A(t)\nabla v)-\lambda J(t)v; \int_Dv^2J(t)\right),
$$
where $J(t):=Jac\Phi(t)$ and $A(t):={D\Phi_t^{-1}}^TD\Phi_t^{-1}\circ\Phi_tJ(t)$. Note that $u$ is an $L^2$-normalized eigenfunction associated with $\lambda$ over $\Omega_t$ if and only if $\mathcal{F}(t,u\circ\Phi_t,\lambda)=(0,1)$. The functional $\mathcal{F}$ is well defined since $J(t)$ belongs to $W^{k-2,\infty}_{loc}$ and $A(t)$ to $W^{k-1,\infty}_{loc}$. Furthermore, it is $C^1$ since the dependence of $\Phi_t$ and $\Phi_t^{-1}$, together with their derivatives, is $C^1$ with respect to $t$. The differential at $(t_0,v_0,\lambda_0)$ with respect to $(v,\lambda)$ is
$$
D\mathcal{F}_{t_0,v_0,\lambda_0}(v,\lambda)=\left(-\dv(A(t_0)\nabla v)-\lambda_0J(t_0)v-\lambda J(t_0) v_0; 2\int_Dvv_0J(t_0)\right)
$$
We need to argue that when $\lambda_0=\lambda$ and $v_0=u_0\circ\Phi_{t_0}$, $u_0$ being an $L^2$ normalized eigenfunction associated with $\lambda$ on $\Omega_{t_0}$, the differential defines an isomorphism from $H_0^1\cap H^k(D)\times\R$ to $H^{k-2}(D)\times\R$. This amounts to prove that for any $f\in H^{k-2}(D)$ and any $c\in\R$, there exists a unique pair $\lambda\in\R$ and $v\in H_0^1\cap H^{k}(D)$ such that
$$
-\dv(A(t_0)\nabla v)-\lambda_0J(t_0)v-\lambda J(t_0) v_0=f,
$$
together with the constraint $\int_Dvv_0J(t_0)=c$. By \cite[Lemme~5.7.3]{henrot-pierre}, which is based on Fredholm's Theorem, there exists a unique solution $v\in H_0^1(D)$ satisfying the integral constraint. Furthermore, since $f+\lambda J(t_0)v_0\in H^{k-2}(D)$ and $A(t_0)\in C^{k-2}(\overline{D})$, by applying elliptic regularity results (e.g. \cite[Theorem~2.20]{gazzola-grunau-sweers}) in a recursive way we conclude that $v\in H^k(D)$, which concludes. As a consequence we can apply the implicit function Theorem to $\mathcal{F}$ which provides a $C^1$ path $t\mapsto (\lambda_t,v_t)\in\R\times H^k(D)$, defined in a neighborhood of $t_0$, where $\lambda_t$ is an eigenvalue over $\Omega_t$ and $v_t\circ\Phi_t^{-1}$ an associated eigenfunction. This proves the lemma.
\end{proof}

%% file: biblio.bib
@article {alessandrini,
    AUTHOR = {Alessandrini, Giovanni},
     TITLE = {Nodal lines of eigenfunctions of the fixed membrane problem in
              general convex domains},
   JOURNAL = {Comment. Math. Helv.},
  FJOURNAL = {Commentarii Mathematici Helvetici},
    VOLUME = {69},
      YEAR = {1994},
    NUMBER = {1},
     PAGES = {142--154},
      ISSN = {0010-2571,1420-8946},
   MRCLASS = {35P05 (35J05)},
  MRNUMBER = {1259610},
MRREVIEWER = {Gianfranco\ Bottaro},
       DOI = {10.1007/BF02564478},
       URL = {https://doi.org/10.1007/BF02564478},
}

@article{burenkov-lamberti-lanza_de_cristoforis,
 author = {Burenkov, V. I. and Lamberti, P. D. and Lanza de Cristoforis, M.},
 title = {Spectral stability of nonnegative self-adjoint operators},
 fjournal = {Journal of Mathematical Sciences (New York)},
 journal = {J. Math. Sci., New York},
 issn = {1072-3374},
 volume = {149},
 number = {4},
 pages = {1417--1452},
 year = {2008},
 language = {English},
 doi = {10.1007/s10958-008-0074-4},
 zbMATH = {8020055},
 Zbl = {1558.35153}
}

@article {camilli,
    AUTHOR = {Camilli, F.},
     TITLE = {A note on convergence of level sets},
   JOURNAL = {Z. Anal. Anwendungen},
  FJOURNAL = {Zeitschrift f\"ur Analysis und ihre Anwendungen. Journal for
              Analysis and its Applications},
    VOLUME = {18},
      YEAR = {1999},
    NUMBER = {1},
     PAGES = {3--12},
      ISSN = {0232-2064,1661-4534},
   MRCLASS = {28A20 (46E35)},
  MRNUMBER = {1681839},
MRREVIEWER = {Aljo\v sa\ Vol\v ci\v c},
       DOI = {10.4171/ZAA/865},
       URL = {https://doi.org/10.4171/ZAA/865},
}

@article {dgh21,
    AUTHOR = {Dahne, Joel and G\'omez-Serrano, Javier and Hou, Kimberly},
     TITLE = {A counterexample to {P}ayne's nodal line conjecture with few
              holes},
   JOURNAL = {Commun. Nonlinear Sci. Numer. Simul.},
  FJOURNAL = {Communications in Nonlinear Science and Numerical Simulation},
    VOLUME = {103},
      YEAR = {2021},
     PAGES = {Paper No. 105957, 13},
      ISSN = {1007-5704,1878-7274},
   MRCLASS = {35P15},
  MRNUMBER = {4291475},
       DOI = {10.1016/j.cnsns.2021.105957},
       URL = {https://doi.org/10.1016/j.cnsns.2021.105957},
}

@article {d00,
    AUTHOR = {Damascelli, Lucio},
     TITLE = {On the nodal set of the second eigenfunction of the
              {L}aplacian in symmetric domains in {$\Bbb R^N$}},
   JOURNAL = {Atti Accad. Naz. Lincei Cl. Sci. Fis. Mat. Natur. Rend. Lincei
              (9) Mat. Appl.},
  FJOURNAL = {Atti della Accademia Nazionale dei Lincei. Classe di Scienze
              Fisiche, Matematiche e Naturali. Rendiconti Lincei. Serie IX.
              Matematica e Applicazioni},
    VOLUME = {11},
      YEAR = {2000},
    NUMBER = {3},
     PAGES = {175--181},
      ISSN = {1120-6330,1720-0768},
   MRCLASS = {35P05 (35A30)},
  MRNUMBER = {1841691},
MRREVIEWER = {I.\ Dan\ Coroian},
}

@article{daners,
 author = {Daners, Daniel},
 title = {Dirichlet problems on varying domains.},
 fjournal = {Journal of Differential Equations},
 journal = {J. Differ. Equations},
 issn = {0022-0396},
 volume = {188},
 number = {2},
 pages = {591--624},
 year = {2003},
 language = {English},
 doi = {10.1016/S0022-0396(02)00105-5},
 zbMATH = {1902862},
 Zbl = {1090.35069}
}

@article {fournais01,
    AUTHOR = {Fournais, S\o ren},
     TITLE = {The nodal surface of the second eigenfunction of the
              {L}aplacian in {${\bf R}^D$} can be closed},
   JOURNAL = {J. Differential Equations},
  FJOURNAL = {Journal of Differential Equations},
    VOLUME = {173},
      YEAR = {2001},
    NUMBER = {1},
     PAGES = {145--159},
      ISSN = {0022-0396,1090-2732},
   MRCLASS = {35P05 (35J05)},
  MRNUMBER = {1836248},
MRREVIEWER = {Vitaly\ A.\ Volpert},
       DOI = {10.1006/jdeq.2000.3868},
       URL = {https://doi.org/10.1006/jdeq.2000.3868},
}

@article {Freitas02,
    AUTHOR = {Freitas, Pedro},
     TITLE = {Closed nodal lines and interior hot spots of the second
              eigenfunction of the {L}aplacian on surfaces},
   JOURNAL = {Indiana Univ. Math. J.},
  FJOURNAL = {Indiana University Mathematics Journal},
    VOLUME = {51},
      YEAR = {2002},
    NUMBER = {2},
     PAGES = {305--316},
      ISSN = {0022-2518,1943-5258},
   MRCLASS = {58J50 (35P05)},
  MRNUMBER = {1909291},
MRREVIEWER = {Vladimir\ Tulovsky},
       DOI = {10.1512/iumj.2002.51.2208},
       URL = {https://doi.org/10.1512/iumj.2002.51.2208},
}

@article {fk08,
    AUTHOR = {Freitas, Pedro and Krej\v ci\v r\'ik, David},
     TITLE = {Location of the nodal set for thin curved tubes},
   JOURNAL = {Indiana Univ. Math. J.},
  FJOURNAL = {Indiana University Mathematics Journal},
    VOLUME = {57},
      YEAR = {2008},
    NUMBER = {1},
     PAGES = {343--375},
      ISSN = {0022-2518,1943-5258},
   MRCLASS = {35J05 (35P15 47F05)},
  MRNUMBER = {2400260},
MRREVIEWER = {Valeri\ S.\ Serov},
       DOI = {10.1512/iumj.2008.57.3170},
       URL = {https://doi.org/10.1512/iumj.2008.57.3170},
}

@article {freitas-krejcirik07,
    AUTHOR = {Freitas, Pedro and Krej\v ci\v r\'ik, David},
     TITLE = {Unbounded planar domains whose second nodal line does not
              touch the boundary},
   JOURNAL = {Math. Res. Lett.},
  FJOURNAL = {Mathematical Research Letters},
    VOLUME = {14},
      YEAR = {2007},
    NUMBER = {1},
     PAGES = {107--111},
      ISSN = {1073-2780},
   MRCLASS = {35J05 (35B05)},
  MRNUMBER = {2289624},
       DOI = {10.4310/MRL.2007.v14.n1.a9},
       URL = {https://doi.org/10.4310/MRL.2007.v14.n1.a9},
}

@book{gazzola-grunau-sweers,
 author = {Gazzola, Filippo and Grunau, Hans-Christoph and Sweers, Guido},
 title = {Polyharmonic boundary value problems. {Positivity} preserving and nonlinear higher order elliptic equations in bounded domains},
 fseries = {Lecture Notes in Mathematics},
 series = {Lect. Notes Math.},
 issn = {0075-8434},
 volume = {1991},
 isbn = {978-3-642-12244-6; 978-3-642-12245-3},
 year = {2010},
 publisher = {Berlin: Springer},
 language = {English},
 doi = {10.1007/978-3-642-12245-3},
 zbMATH = {5712793},
 Zbl = {1239.35002}
}

@misc{ha24,
	author = {Ha, Soo Sun},
	title = {On the nodal lines of eigenfunctions of {L}aplacian in plane},
	date = {2024},
	note = {Preprint at https://vixra.org/pdf/2312.0050v2.pdf}
}

@article {helffer-hoffmann-ostenhof-jauberteau-lena,
    AUTHOR = {Helffer, B. and Hoffmann-Ostenhof, T. and Jauberteau, F. and
              L\'ena, C.},
     TITLE = {On the multiplicity of the second eigenvalue of the
              {L}aplacian in non simply connected domains---with some
              numerics},
   JOURNAL = {Asymptot. Anal.},
  FJOURNAL = {Asymptotic Analysis},
    VOLUME = {121},
      YEAR = {2021},
    NUMBER = {1},
     PAGES = {35--57},
      ISSN = {0921-7134,1875-8576},
   MRCLASS = {35P05 (35J05)},
  MRNUMBER = {4183455},
MRREVIEWER = {Jolanta\ Przybycin},
       DOI = {10.3233/asy-191594},
       URL = {https://doi.org/10.3233/asy-191594},
}

@book{henrot-pierre,
 author = {Henrot, Antoine and Pierre, Michel},
 title = {Variation et optimisation de formes. {Une} analyse g{\'e}om{\'e}trique},
 fseries = {Math{\'e}matiques \& Applications (Berlin)},
 series = {Math. Appl. (Berl.)},
 issn = {1154-483X},
 volume = {48},
 isbn = {3-540-26211-3},
 year = {2005},
 publisher = {Berlin: Springer},
 language = {French},
 doi = {10.1007/3-540-37689-5},
 zbMATH = {2191966},
 Zbl = {1098.49001}
}

@article {H2ON,
    AUTHOR = {Hoffmann-Ostenhof, M. and Hoffmann-Ostenhof, T. and
              Nadirashvili, N.},
     TITLE = {The nodal line of the second eigenfunction of the {L}aplacian
              in {${\bf R}^2$} can be closed},
   JOURNAL = {Duke Math. J.},
  FJOURNAL = {Duke Mathematical Journal},
    VOLUME = {90},
      YEAR = {1997},
    NUMBER = {3},
     PAGES = {631--640},
      ISSN = {0012-7094,1547-7398},
   MRCLASS = {35P15 (35J05)},
  MRNUMBER = {1480548},
MRREVIEWER = {B.\ Hellwig},
       DOI = {10.1215/S0012-7094-97-09017-7},
       URL = {https://doi.org/10.1215/S0012-7094-97-09017-7},
}

@incollection {H2ON-2,
    AUTHOR = {Hoffmann-Ostenhof, M. and Hoffmann-Ostenhof, T. and
              Nadirashvili, N.},
     TITLE = {On the nodal line conjecture},
 BOOKTITLE = {Advances in differential equations and mathematical physics
              ({A}tlanta, {GA}, 1997)},
    SERIES = {Contemp. Math.},
    VOLUME = {217},
     PAGES = {33--48},
 PUBLISHER = {Amer. Math. Soc., Providence, RI},
      YEAR = {1998},
      ISBN = {0-8218-0861-3},
   MRCLASS = {35J25 (35B99 35P99)},
  MRNUMBER = {1605269},
MRREVIEWER = {Vitaly\ A.\ Volpert},
       DOI = {10.1090/conm/217/02980},
       URL = {https://doi.org/10.1090/conm/217/02980},
}

@incollection {jerison,
    AUTHOR = {Jerison, David},
     TITLE = {The first nodal set of a convex domain},
 BOOKTITLE = {Essays on {F}ourier analysis in honor of {E}lias {M}. {S}tein
              ({P}rinceton, {NJ}, 1991)},
    SERIES = {Princeton Math. Ser.},
    VOLUME = {42},
     PAGES = {225--249},
 PUBLISHER = {Princeton Univ. Press, Princeton, NJ},
      YEAR = {1995},
      ISBN = {0-691-08655-9},
   MRCLASS = {35P15 (35B05 35J20)},
  MRNUMBER = {1315550},
MRREVIEWER = {Fabio\ Cipriani},
}

@article {kennedy,
    AUTHOR = {Kennedy, J. B.},
     TITLE = {Closed nodal surfaces for simply connected domains in higher
              dimensions},
   JOURNAL = {Indiana Univ. Math. J.},
  FJOURNAL = {Indiana University Mathematics Journal},
    VOLUME = {62},
      YEAR = {2013},
    NUMBER = {3},
     PAGES = {785--798},
      ISSN = {0022-2518,1943-5258},
   MRCLASS = {35P05 (35J05 58J05)},
  MRNUMBER = {3164844},
MRREVIEWER = {Alexander\ A.\ Pankov},
       DOI = {10.1512/iumj.2013.62.4975},
       URL = {https://doi.org/10.1512/iumj.2013.62.4975},
}

@article {kiwan,
    AUTHOR = {Kiwan, Rola},
     TITLE = {On the nodal set of a second {D}irichlet eigenfunction in a
              doubly connected domain},
   JOURNAL = {Ann. Fac. Sci. Toulouse Math. (6)},
  FJOURNAL = {Annales de la Facult\'e{} des Sciences de Toulouse.
              Math\'ematiques. S\'erie 6},
    VOLUME = {27},
      YEAR = {2018},
    NUMBER = {4},
     PAGES = {863--873},
      ISSN = {0240-2963,2258-7519},
   MRCLASS = {35P15 (35J05 35P05 49R05)},
  MRNUMBER = {3884612},
MRREVIEWER = {Srinivasan\ Kesavan},
       DOI = {10.5802/afst.1585},
       URL = {https://doi.org/10.5802/afst.1585},
}

@article {LinNi,
    AUTHOR = {Lin, Chang Shou and Ni, Wei-Ming},
     TITLE = {A counterexample to the nodal domain conjecture and a related
              semilinear equation},
   JOURNAL = {Proc. Amer. Math. Soc.},
  FJOURNAL = {Proceedings of the American Mathematical Society},
    VOLUME = {102},
      YEAR = {1988},
    NUMBER = {2},
     PAGES = {271--277},
      ISSN = {0002-9939,1088-6826},
   MRCLASS = {35B05 (35J60)},
  MRNUMBER = {920985},
MRREVIEWER = {C.\ A.\ Swanson},
       DOI = {10.2307/2045874},
       URL = {https://doi.org/10.2307/2045874},
}

@article {linqun,
    AUTHOR = {Liqun, Zhang},
     TITLE = {On the multiplicity of the second eigenvalue of {L}aplacian in
              {${\bf R}^2$}},
   JOURNAL = {Comm. Anal. Geom.},
  FJOURNAL = {Communications in Analysis and Geometry},
    VOLUME = {3},
      YEAR = {1995},
    NUMBER = {1-2},
     PAGES = {273--296},
      ISSN = {1019-8385,1944-9992},
   MRCLASS = {35P05 (35J05)},
  MRNUMBER = {1362653},
       DOI = {10.4310/CAG.1995.v3.n2.a3},
       URL = {https://doi.org/10.4310/CAG.1995.v3.n2.a3},
}

@article {Melas,
    AUTHOR = {Melas, Antonios D.},
     TITLE = {On the nodal line of the second eigenfunction of the
              {L}aplacian in {${\bf R}^2$}},
   JOURNAL = {J. Differential Geom.},
  FJOURNAL = {Journal of Differential Geometry},
    VOLUME = {35},
      YEAR = {1992},
    NUMBER = {1},
     PAGES = {255--263},
      ISSN = {0022-040X,1945-743X},
   MRCLASS = {35P05 (35J05 58G25)},
  MRNUMBER = {1152231},
       URL = {http://projecteuclid.org/euclid.jdg/1214447811},
}

@article {mukherjee-saha,
    AUTHOR = {Mukherjee, Mayukh and Saha, Soumyajit},
     TITLE = {On the effects of small perturbation on low energy {L}aplace
              eigenfunctions},
   JOURNAL = {J. Spectr. Theory},
  FJOURNAL = {Journal of Spectral Theory},
    VOLUME = {15},
      YEAR = {2025},
    NUMBER = {3},
     PAGES = {1045--1087},
      ISSN = {1664-039X,1664-0403},
   MRCLASS = {58J50 (35J05)},
  MRNUMBER = {4949372},
       DOI = {10.4171/jst/570},
       URL = {https://doi.org/10.4171/jst/570},
}

@article {Payne,
    AUTHOR = {Payne, L. E.},
     TITLE = {Isoperimetric inequalities and their applications},
   JOURNAL = {SIAM Rev.},
  FJOURNAL = {SIAM Review. A Publication of the Society for Industrial and
              Applied Mathematics},
    VOLUME = {9},
      YEAR = {1967},
     PAGES = {453--488},
      ISSN = {0036-1445},
   MRCLASS = {52.40 (35.00)},
  MRNUMBER = {218975},
MRREVIEWER = {J.\ Hersch},
       DOI = {10.1137/1009070},
       URL = {https://doi.org/10.1137/1009070},
}

@article {Payne2,
    AUTHOR = {Payne, Lawrence E.},
     TITLE = {On two conjectures in the fixed membrane eigenvalue problem},
   JOURNAL = {Z. Angew. Math. Phys.},
  FJOURNAL = {Zeitschrift f\"ur Angewandte Mathematik und Physik. ZAMP.
              Journal of Applied Mathematics and Physics. Journal de
              Math\'ematiques et de Physique Appliqu\'ees},
    VOLUME = {24},
      YEAR = {1973},
     PAGES = {721--729},
      ISSN = {0044-2275,1420-9039},
   MRCLASS = {35P99},
  MRNUMBER = {333487},
MRREVIEWER = {J.\ Hersch},
       DOI = {10.1007/BF01597076},
       URL = {https://doi.org/10.1007/BF01597076},
}

@incollection {yau93,
    AUTHOR = {Yau, Shing-Tung},
     TITLE = {Open problems in geometry},
 BOOKTITLE = {Differential geometry: partial differential equations on
              manifolds ({L}os {A}ngeles, {CA}, 1990)},
    SERIES = {Proc. Sympos. Pure Math.},
    VOLUME = {54, Part 1},
     PAGES = {1--28},
 PUBLISHER = {Amer. Math. Soc., Providence, RI},
      YEAR = {1993},
      ISBN = {0-8218-1494-X},
   MRCLASS = {53-02},
  MRNUMBER = {1216573},
       DOI = {10.1090/pspum/054.1/1216573},
       URL = {https://doi.org/10.1090/pspum/054.1/1216573},
}
